\documentclass[10pt]{article}
\font\smallit=cmti10

\usepackage{amssymb,latexsym,amsmath,epsfig,amsthm} 
\usepackage{mathtools}

\makeatletter

\renewcommand\section{\@startsection {section}{1}{\z@}
{-30pt \@plus -1ex \@minus -.2ex}
{2.3ex \@plus.2ex}
{\normalfont\normalsize\bfseries}}

\renewcommand\subsection{\@startsection{subsection}{2}{\z@}
{-3.25ex\@plus -1ex \@minus -.2ex}
{1.5ex \@plus .2ex}
{\normalfont\normalsize\bfseries}}

\renewcommand{\@seccntformat}[1]{\csname the#1\endcsname. }

\makeatother

\numberwithin{equation}{section}
\newtheorem{prop}{Proposition}

\newtheorem{thm}{Theorem}
\newtheorem{lem}{Lemma}
\newtheorem{col}{Corollary}
\numberwithin{thm}{section}
\numberwithin{prop}{section}
\numberwithin{col}{thm}
\numberwithin{lem}{section}

\begin{document}

\begin{center}
\uppercase{\bf Extended Congruences for Harmonic Numbers}
\vskip 20pt
{\bf Ren\'{e} Gy}\\
{\smallit (Bondy, France)}\\
{\tt rene.gy@numericable.com}\\
\vskip 10pt
\end{center}
\vskip 30pt

\vskip 30pt

\centerline{\bf Abstract}
\noindent
We derive $p$-adic expansions for the generalized Harmonic numbers $H^{(j)}_{p-1}$ and $H^{(j)}_{\frac{p-1}{2}}$ involving the Bernoulli numbers $B_j$  and the the base-2 Fermat quotient $q_p$. While most of our results are not new, we obtain them elementarily, without resorting to the theory of $p$-adic L-functions as was the case previously. Moreover, we show that  
\begin{equation*}\sum_{j=0}^{n-1}\left(\frac{(2^{j+1}-1)}{(j+1)}\frac{(2^{j+2}-1)}{(j+2)}\frac{B_{j+2}}{2^{j}}H^{(j+1)}_{\frac{p-1}{2}}+2(-1)^j\frac{q_p^{j+1}}{j+1}\right)p^j\equiv 0 \pmod {p^n} \end{equation*} holds under the condition that $p >\frac{n+1}{2}$. This is another generalization, modulo any prime power, of the old $p$-congruence  $H_{\frac{p-1}{2}}+2q_p \equiv 0 \bmod p$ attributed to Eisenstein, which is stronger than the one which has been published recently \cite{Lin18}.

\pagestyle{myheadings} 

\thispagestyle{empty} 
\baselineskip=12.875pt 
\vskip 30pt
\section{Introduction}
It is well-known that the Wolstenholme theorem can be stated as a congruence for the {\it Harmonic number} $H_{p-1}$ where $p \ge5$ is a prime number:  

\begin{equation}\label{1.1}H_{p-1}:=\sum_{j=1}^{p-1}\frac{1}{j} \equiv 0 \pmod{p^2}.\end{equation} 

\noindent There also exists (\cite{Mestrovic13}, Lemma 14) a refinement of  \eqref{1.1} valid for $p \ge7$: 
\begin{equation}\label{1.2}H_{p-1} + \frac{p}{2}H^{(2)}_{p-1}  \equiv 0 \pmod{p^4},\end{equation} 
\noindent where $H^{(m)}_{n}:=\sum_{j=1}^{n}\frac{1}{j^m}$ is the {\it generalized Harmonic number}, ($H^{(1)}_n=H_{n}$).\\
\noindent Another well-known congruence, due to Eisenstein \cite{Eisenstein1850}, reads 
\begin{equation}\label{1.3}H_\frac{p-1}{2} + 2\cdot q_p  \equiv 0 \pmod{p},\end{equation}
\noindent where $p\ge3$ is a prime number and $q_p=\frac{2^{p-1}-1}{p}$ is the base-$2$  {\it Fermat quotient}.\\
\noindent It happens that this congruence may also be refined \cite{Lehmer38} in the following way:
\begin{equation}\label{1.4}H_\frac{p-1}{2} + 2\cdot q_p-p\cdot q_p^2  \equiv 0 \pmod{p^2}.\end{equation}
\noindent Generalizations of these congruences modulo $p^n$ for arbitrary large $n$ can be obtained from the theory of $p$-adic L-function \cite{Washington98}, \cite{Lin18}. In the present paper, we will show how such generalizations are also obtained more elementarily and we will provide sufficient conditions on $p$ ensuring their validity for even higher powers of $p$. \\
\noindent Notations and useful Lemmas are presented in Section 2. Section 3 is devoted to the cases of $H^{(j)}_{p-1}$ and $H^{(2j)}_{\frac{p-1}{2}}$. Section 4 deals with the more difficult case of $H^{(j)}_{\frac{p-1}{2}}$ for odd $j$.

\section{Notations and useful preliminaries}
In this paper, $g, h, i, j, k, n, m$ are non-negative integers, $p$ is a prime number and $x$ the argument of a generating function. In addition to the notations already given in the introduction for the Harmonic numbers and the base-$2$ Fermat quotient, we let $\delta^j_i=1$ when $i=j$ and  $\delta^j_i=0$ otherwise, we let $\binom{n}{k}$ be the usual binomial coefficient, $B_n$ a {\it Bernouilli number} ($B_0=1$, $B_1=-1$, $B_2=\frac{1}{6}$, ..), and $\mathbf{i}^2=-1$. The following classical results will be used and are stated without proof.\\
(i) We have the generalized Binomial Theorem:  
\begin{equation}\label{e2a} \frac{1}{(1-x)^k} =\sum_{j\ge0}\binom{j+k-1}{j}x^j.\end{equation}
(ii) Let $ n\ge 1$, and $1\le j\le p-1$, we have the Euler Theorem:
\begin{equation}\label{} j^{p^{n-1}(p-1)} \equiv 1 \pmod{p^n}.\end{equation}
(iii) We will use the Legendre formula which states that the highest power of $p$ which divides $j!$ is ${\frac{j-s_p(j)}{p-1}}$, where $s_p(j)$ is the sum of the base-$p$ digits of $j$.\\
(iv) We shall also make use of some classical properties of the Bernoulli numbers.\\
They have the following exponential generating function: 
\begin{align}
\label{e2bb} \frac{x}{e^x-1}&=\sum_{n\ge0}B_n\frac{x^n}{n!},\end{align}
and they obey the following recurrence relation:
\begin{align}
\label{e2b} \sum_{0\le k\le n}B_k\binom{n}{k} &= (-1)^n B_n.\end{align}
They vanish at odd indices, larger than $1$, so that
\begin{align}\label{e2bbis} B_{2n+1}&=0  \text { \ \ \ \  when } n\ge 1, \end{align}
and we have the Faulhaber formula for the sum of consecutive $i$-th powers: 
\begin{align}\label{e2d} \sum_{j=1}^n j^i&=\frac{1}{i+1}\sum_{h=0}^i(-1)^h\binom{i+1}{h} B_h n^{i+1-h}.\end{align}
We will also need the Kummer congruence: let $h,k$ not divisible by $p-1$, such that $ h\equiv k \pmod{p-1} $, we have
\begin{equation}\label{e2ez} \frac{B_h}{h}-\frac{B_k}{k} \equiv 0 \pmod {p}.\end{equation}
\noindent We will make use of the Von Staudt-Clausen theorem: let $D(B_{2j})$ be the denominator of $B_{2j}$ in the reduced form, we have
\begin {equation}\label{e2e} D(B_{2j})=\prod_{p-1\vert 2j}p. \end{equation}
\noindent Finally, recall that $(p, 2k)$ is an {\it irregular pair} when $p \ge 2k+3$ and $B_{2k}\equiv 0 \pmod p$.\\
\ \\
\noindent In addition to the above classical results stated without proof, we shall make use of some lemmas, given hereafter with their proofs for the sake of self-containment.

\begin{lem}\label{lem1}  Let $k$ be a positive natural integer, we have the following identities:
\begin{align}\label{azd} \sum_{j=1}^k \binom{2k-1}{2j-1}B_{2j} &=\frac{1}{2}+B_{2k}+B_{2k-1},\\
\label{azd1} \sum_{j=1}^k \binom{2k}{2j-1}B_{2j} &=\frac{1}{2}-B_{2k},\\
\label{azdd} \sum_{j=1}^k \binom{2k+1}{2j-1}B_{2j} &=\frac{1}{2},\\
\label{azdd1} \sum_{j=1}^k \binom{2k+2}{2j-1}B_{2j} &=\frac{1}{2}-(2k+3)B_{2k+2}.\end{align}
\begin{proof} The proofs only make use of  \eqref{e2b}. We have
\begin{align*}
 \sum_{j=1}^k \binom{2k-1}{2j-1}B_{2j} =& \sum_{j=1}^k \binom{2k}{2j}B_{2j} - \sum_{j=1}^k \binom{2k-1}{2j}B_{2j} \\
=& \sum_{j=0}^{2k} \binom{2k}{j}B_{j}-1+k -\sum_{j=0}^{2k-1} \binom{2k-1}{j}B_{j}+1-\frac{1}{2}(2k-1)\\
=& B_{2k}+B_{2k-1}+\frac{1}{2},\\
\sum_{j=1}^k  \binom{2k}{2j-1}B_{2j} =& \sum_{j=1}^k  \binom{2k+1}{2j}B_{2j} - \sum_{j=1}^k  \binom{2k}{2j}B_{2j} \\
=& \sum_{j=0}^{2k+1}  \binom{2k+1}{j}B_{j}-1+\frac{2k+1}{2}-B_{2k+1}\\
&-\sum_{j=0}^{2k}  \binom{2k}{j}B_{j}+1-(2k)\frac{1}{2} =\frac{1}{2}-B_{2k},\\
\sum_{j=1}^k\binom{2k+1}{2j-1}B_{2j} =&\sum_{j=1}^{k+1}\binom{2k+1}{2j-1}B_{2j}-B_{2k+2}\\
=&B_{2k+2}+B_{2k+1}+\frac{1}{2}-B_{2k+2}=\frac{1}{2},\\
\sum_{j=1}^k \binom{2k+2}{2j-1}B_{2j} =&\sum_{j=1}^{k+1} \binom{2k+2}{2j-1}B_{2j}-(2k+2)B_{2k+2}\\
=& \frac{1}{2}-B_{2k+2}-(2k+2)B_{2k+2}=\frac{1}{2}-(2k+3)B_{2k+2}.
\end{align*}
\end{proof}
\end{lem} 

\begin{lem}\label{ww} Let $k\ge0$ be an integer, it holds that
\begin{align}\label{aze1} \sum_{0\le j}B_j(2^j -1)\binom{k}{j} &= (-1)^k B_k(1-2^k),\\
\label{aze1bis} \sum_{0\le j}B_j2^j \binom{k}{j} &=2B_k(1-2^{k-1}).\end{align}
\begin{proof}
Via the generating function \eqref{e2bb}, we have
\begin{align*}  \sum_{0\le k}B_k(2^k -1)\frac{x^k}{k!} =\sum_{0\le k}B_k\frac{(2x)^k}{k!} -\sum_{0\le k}B_k\frac{x^k}{k!} = \frac{2x}{e^{2x}-1}-\frac{x}{e^x-1}&=\frac{-x}{e^{x}+1}.\end{align*}
Recall \cite{Wilf92} that when $f(x)$ is the exponential generating function of the sequence $a_k$, then $e^xf(x)$ is the exponential generating function of the sequence $b_k=\sum_{0\le j }\binom{k}{j}a_j$. Then
\begin{align*}\sum_{0\le k}\sum_{0\le j}\binom{k}{j}B_j(2^j -1)\frac{x^k}{k!} =e^x\frac{-x}{e^x+1} &=\frac{-x}{e^{-x}+1},\end{align*}
and, on the other hand, from the the first line, we also have
\begin{align*} -\sum_{0\le k}B_k(2^k -1)\frac{(-x)^k}{k!} &=\frac{-x}{e^{-x}+1},
\end{align*} 
and \eqref{aze1} follows from the identification of the coefficients in both series expansions for $\frac{-x}{e^{-x}+1}$. For the proof of \eqref{aze1bis}, it is easy to directly check its validity for $k=0$ and $k=1$, and for $k>1$ it follows from \eqref{aze1} and \eqref{e2bbis}.
\end{proof} 
\end{lem}

\begin{lem}\label{2.3}  Let $k\ge 1$ be an integer, we have
\begin{equation}\label{aze12} \sum_{j=0}^{2k-1}\binom{2k-1}{j}(2^j-1)(2^{j+1}-1)\frac{B_{j+1}}{j+1} = (2^{2k}-1)\frac{B_{2k}}{2k}.\end{equation}
\begin{proof}
We start from
\begin{align*} 
\frac{x}{2} \tan\frac{x}{2}&=\frac{\mathbf{i}x}{2}\frac{e^{-\frac{\mathbf{i}x}{2}}-e^{\frac{\mathbf{i}x}{2}}}{e^{-\frac{\mathbf{i}x}{2}}+e^{\frac{\mathbf{i}x}{2}}}=\frac{\mathbf{i}x}{2}\frac{1-e^{\mathbf{i}x}}{1+e^{\mathbf{i}x}}
=\frac{\mathbf{i}x}{e^{\mathbf{i}x}+1}-\frac{1}{2}\mathbf{i}x\\
&=-\sum_{0\le k}B_k(2^k -1)\frac{(\mathbf{i}x)^k}{k!}-\frac{1}{2}\mathbf{i}x \text{  \ \   from the proof of Lemma \ref{ww}}\\
&=-\sum_{2\le k}B_k(2^k -1)\frac{(\mathbf{i}x)^k}{k!}\\
&=-\sum_{1\le k}B_{2k}(2^{2k} -1)\frac{(\mathbf{i}x)^{2k}}{(2k)!} \text{  \ \   since } B_{k+1}=0 \text{ for even $k\ge 2$ and then}\\
\tan\frac{x}{2}&=2\sum_{1\le k}(-1)^{k-1}(2^{2k} -1)\frac{B_{2k}}{2k}\frac{x^{2k-1}}{(2k-1)!}.
\end{align*}
Then, 
\begin{align*}
\tan x-\tan\frac{x}{2}&= 2\sum_{1\le k}\mathbf{i}^{k-1}(2^k -1)(2^{k+1} -1)\frac{B_{k+1}}{k+1}\frac{x^{k}}{k!},\\
\mathbf{i}\frac{\tan x-\tan\frac{x}{2}}{2}&=\sum_{0\le k}(2^k -1)(2^{k+1} -1)\frac{B_{k+1}}{k+1}\frac{(\mathbf{i}x)^{k}}{k!},\\
\mathbf{i}e^{\mathbf{i}x}\frac{\tan x-\tan\frac{x}{2}}{2}&=\sum_{0\le k}\sum_{0\le j}\binom{k}{j}(2^j -1)(2^{j+1} -1)\frac{B_{j+1}}{j+1}\frac{(\mathbf{i}x)^{k}}{k!}.
\end{align*}
We take the imaginary part, so that
\begin{align*}
\cos x \frac{\tan\frac{x}{2}-\tan x}{2}&=\sum_{1\le k}\sum_{0\le j }\binom{2k-1}{j}(2^j -1)(2^{j+1} -1)\frac{B_{j+1}}{j+1}(-1)^k\frac{x^{2k-1}}{(2k-1)!},\end{align*}
which is
\begin{align*}\tan\frac{x}{2}&=2\sum_{1\le k}\sum_{0\le j }\binom{2k-1}{j}(2^j -1)(2^{j+1} -1)\frac{B_{j+1}}{j+1}(-1)^{k-1}\frac{x^{2k-1}}{(2k-1)!}.
\end{align*}
The desired result is obtained by comparison of the coefficients in both series expansions for $\tan\frac{x}{2}$ which have been obtained in the above derivations. \end{proof} 
\end{lem} 

\begin{lem}\label{binom}   Let $p$ be prime and $n\ge 1$, $i,j\ge0$  integers, we have
$$j!\binom{p^{n-1}(p-1)-i}{j} \equiv (-1)^j j!\binom{i+j-1}{j} \pmod{p^{n-1}}.$$ 
\begin{proof}We write the factors of the binomial coefficient on the left hand side and reduce modulo $p^{n-1}$, so that
\begin{align*}j!\binom{p^{n-1}(p-1)-i}{j} &=(p^{n-1}(p-1)-i)\cdot \cdot \cdot (p^{n-1}(p-1)-i-j+1)\\
&\equiv (-1)^j j!\binom{i+j-1}{j} \pmod{p^{n-1}}.\end{align*}
\end{proof} 
\end{lem}

\begin{lem}\label{lemjohn}   Let $p$ be prime, $j\ge 1$ an integer and $v_p(k)$ the higest power of $p$ which divides the integer $k$, we have
$$v_p\left( \frac{p^{j-1}}{j!}\right) \ge (j-1)\frac{p-2}{p-1}.$$ 
In particular, $p$ divides $\frac{p^{j-1}}{j!}$ when $p$ is odd and $j\ge2$.
\begin{proof}  We reproduce the proof from \cite{Johnson75}. We have
$ v_p\left( \frac{p^{j-1}}{j!}\right)=j-1- v_p\left( j!\right)= j-1 - \frac{j-s_p(j)}{p-1}$, by Legendre formula. But clearly, $s_p(j) \ge 1$, and the claim follows.
\end{proof} 
\end{lem}

\begin{lem}\label{2.5}   Let $p$ be an odd prime and $n\ge 1$, $h \ge0$  integers. We have the two following congruences:
\begin{equation}\label {251} pB_{p^{n-1} (p-1)}\equiv p-1 \pmod {p^n}, \end {equation} 
\begin{equation}\label {252}pB_{p^{n-1}(p-1)-2h} \equiv H^{(2h)}_{p-1} \pmod{p}. \end{equation}
\begin{proof}
We make use of the Faulhaber formula \eqref{e2d} with  $i=p^{n-1} (p-1)-2h=2m$ an even exponent. After some manipulation, accounting for the vanishing of the odd-indexed Bernoulli numbers from $B_3$, we have 
\begin{align*}\label{} \sum_{j=1}^pj^{2m}&=pB_{2m} +2m\sum_{g=1}^m \frac{pB_{2m-2g}\binom{2m-1}{2g-1}}{2g(2g+1)} p^{2g} + \frac{1}{2}p^{2m}.\end{align*}
From Euler theorem, $\sum_{j=1}^pj^{2m} \equiv  H^{(2h)}_{p-1} \bmod p^n $, then
\begin{align*}
pB_{2m}&\equiv  H^{(2h)}_{p-1}-(p^{n-1} (p-1)-2h)\sum_{g=1}^m \frac{pB_{2m-2g}\binom{2m-1}{2g-1}}{2g(2g+1)} p^{2g} + \frac{1}{2}p^{2m} \pmod {p^n}.   \end{align*}
But, when $ n\ge 2h $, $ 2m= p^{n-1} (p-1)-2h\ge p^{n-1} (p-1)-n  \ge n$ as soon as $p\ge 3$. Also $pB_{2m-2g}\binom{2m-1}{2g-1} $ is $p$-integral and, for $g \ge 1$, $\frac{p^{2g}}{2g(2g+1)} $ is divisible by $p$, by Lemma \ref{lemjohn}. Then
\begin{align*}
pB_{p^{n-1} (p-1)-2h}&\equiv H^{(2h)}_{p-1}+2h\sum_{g=1}^m \frac{pB_{2m-2g}\binom{2m-1}{2g-1}}{2g(2g+1)} p^{2g} \pmod {p^n}.\end{align*}
When $h=0$, this readily establishes \eqref{251}, and when $h\neq 0$, we have \eqref{252}.
\end{proof} 
\end{lem} 

\begin{lem}\label{2.6}    
	Let $n\ge 1$ be an integer and  $p$ an odd prime such that $p > \frac{n+1}{2}$. We have the following congruence: 
\begin{equation}\label{e2.5a}  \frac{ 2^{p^{n-1}(p-1)}-1  }{p^{n}} \equiv \sum_{j=0}^{n-1}(-1)^j\frac{q_p^{j+1}}{j+1}p^j +\delta^{n+1}_p q_pp^{n-1}\pmod {p^n}.\end{equation}
\begin{proof} This is mainly taken from Lemma 2.8 in \cite{Sun08}, but our condition on $p$ is less restrictive than in  \cite{Sun08}. We start with
\begin{align*} \frac{ 2^{p^{n-1}(p-1)}-1  }{p^{n}}&= \frac{ (1+p\cdot q_p)^{p^{n-1}}-1}{p^n}= \frac{1}{p^{n}}\sum_{j=1}^{p^{n-1}}\binom{p^{n-1}}{j}p^jq^j_p\\
&=q_p+\sum_{j=2}^{p^{n-1}}(p^{n-1}-1)\cdot \cdot \cdot (p^{n-1}-j+1)\frac{p^{j-1}}{j!}q^j_p.
\end{align*}
By Lemma \ref{lemjohn}, we have $ \frac{p^{j-1}}{j!}\equiv 0 \pmod p$, when $j\ge2$, then
\begin{align*} \frac{ 2^{p^{n-1}(p-1)}-1  }{p^{n}}&  \equiv  q_p+\sum_{j=2}^{p^{n-1}}(-1)\cdot \cdot \cdot (-j+1)\frac{p^{j-1}}{j!}q^j_p \pmod {p^{n}}\\
&\equiv  \sum_{j=1}^{p^{n-1}}(-1)^{j-1}\frac{p^{j-1}}{j}q^j_p \\
&\equiv  \sum_{j=0}^{n-1}(-1)^{j}\frac{p^{j}}{j+1}q^{j+1}_p +p^{n} \sum_{j=n+1}^{p^{n-1}}(-1)^{j-1}\frac{p^{j-(n+1)}}{j}q^j_p\\
&\equiv  \sum_{j=0}^{n-1}(-1)^{j}\frac{p^{j}}{j+1}q^{j+1}_p +p^{n}\sum_{\underset{j\equiv0\bmod p}{j=n+1}}^{p^{n-1}} (-1)^{j-1}\frac{p^{j-(n+1)}}{j}q^j_p \pmod {p^{n}}.
\end{align*}
To finish the proof, we consider now the summands in the second term of the latter congruence and we show that they are all $p$-integral, except when $j=n+1=p$.
Since $j \equiv 0\bmod p$, let $m >0$ be the highest integer such that $j=k\cdot p^m $. We have $j-(n+1)-m = k\cdot p^m -(n+1)-m  \ge p^m -(n+1)-m$. Suppose first that $m\ge2$. Then we have $p^m -m -(n+1)\ge 2p-1-(n+1) \ge 0$ since $p^m -m \ge 2p-1$ when  $m \ge 2, \ p\ge 3$, and since  $p\ge \frac {n+2}{2}$ by hypothesis. Then the corresponding summands in second term on the right hand side in the latter congruence are $p$-integral.
Suppose now that $m=1$. Then, we have
 \begin{align*}j-(n+1)-m = k\cdot p -(n+1)-1 &= (k-1)p +p -(n+1)-1\\&\ge (k-1)\frac {n+2}{2} -\frac {n+2}{2} \ge 0 \text{, if $k \ge 2$.} \end{align*}
Then, when $m=1, k\ge2$, the corresponding summands in second term on the right hand side in the latter congruence are also $p$-integral. There remains the case $m=k=1$. Then, $j=p$ and $ j-(n+1)-m= j-(n+2) \ge 0$ if $j \ge n+2$, and then the only case which is left is $j=n+1=p$. Finally, the last term on the right hand side of \eqref{e2.5a} is justified, since $q_p^p \equiv q_p \bmod p$.\\
	\end{proof} 
	\end{lem} 
	
\section{Extended congruences for $H^{(j)}_{p-1}$  and $H^{(2j)}_{\frac{p-1}{2}}$   ($j\ge 1$) }
We present now a first kind of $p$-adic expansions for $H^{(j)}_{p-1}$ and $H^{(2j)}_{\frac{p-1}{2}}$ ($j\ge 1$) in terms of higher order Harmonic numbers and a second one, with higher order Harmonic numbers as well, but also involving Bernoulli numbers. These results are not new, except possibly for Theorem \ref{H2kp-1/2}. Hereafter, they are obtained with elementary algebric and arithmetic arguments and with the classical Falhauber formula \eqref{e2d} for sums of consecutive powers in terms of Bernoulli numbers.
\begin{thm}\label{prop4}    
Let  $ k \ge 1$ an integer and $p$ a prime number, the generalized Harmonic numbers $H^{(k)}_{p-1}$ is expanded in the following  p-adically converging sum of powers of $p$:
\begin{align}\label{e10ee}   H^{(k)}_{p-1}&=(-1)^k \sum_{j\ge0} \binom {j+k-1}{j}H^{(k+j)}_{p-1}p^j.\end{align}
Moreover, when $p$ is odd, we have
\begin{align}\label{e10eed}H^{(2k)}_{p-1}&=2H^{(2k)}_{\frac{p-1}{2}}+ \sum_{j\ge 1} \binom {j+2k-1}{j}H^{(2k+j)}_{\frac{p-1}{2}}p^j,\\
\label{e10eee}   H^{(k)}_{p-1}&= \frac{1+(-1)^k}{2^{k}}H^{(k)}_{\frac{p-1}{2}}+ (-1)^k \sum_{j\ge1} \binom {j+k-1}{j}\frac{1}{2^{j+k}}H^{(k+j)}_{\frac{p-1}{2}}p^j,\\
\label{e10eeeff}  2(2^{2k}-1) H^{(2k)}_{\frac{p-1}{2}}&=-\sum_{j\ge1} \binom {j+2k-1}{j}\frac{2^{2k+j}-1}{2^j}H^{(2k+j)}_{\frac{p-1}{2}}p^j.\end{align}
\begin{col}\label{Remark0}  
From these series, one readily obtains the following congruences, valid but for odd $p$:\\

\begin{align}\label{e10eeez}H^{(2k-1)}_{p-1}&\equiv 0 \pmod {p},\\
\label{e10eeezz}H^{(2k-1)}_{p - 1}+p(2k-1)H^{(2k)}_{\frac{p-1}{2}}&\equiv 0 \pmod{p^2},\\
\label{e9a}   H^{(2k-1)}_{p-1}+\frac{1}{2}p(2k-1)H^{(2k)}_{p-1}  &\equiv 0 \pmod {p^2},  \\
\label{e10eeea}H^{(2k)}_{p-1}-2H^{(2k)}_{\frac{p-1}{2}}- p\cdot 2kH^{(2k+1)}_{\frac{p-1}{2}} &\equiv 0\pmod {p^2},\\
\label{e10eeee} H^{(2k)}_{\frac{p-1}{2}}(2^{2k}-1)+p\cdot \frac{k}{2}(2^{2k+1} -1)H^{(2k+1)}_{\frac{p-1}{2}}&\equiv 0 \pmod{p^2},\\
\label{e10eeeb} 2H^{(2k)}_{\frac{p-1}{2}}-(2^{2k+1} -1)H^{(2k)}_{p-1} &\equiv 0\pmod{p^2}.
\end{align}
\end{col} \end{thm}
\noindent{\bf Remark.} In the case $k=1$ \eqref{e10ee} and \eqref{e10eeeff} reads, respectively,
\begin{align}\label{e10eex}   H_{p-1}+\frac{1}{2}\sum_{j\ge1} H^{(j+1)}_{p-1}p^j &=0,\\
\label{e10eeeffx} H^{(2)}_{\frac{p-1}{2}}+\sum_{j\ge1} \frac{j+1}{6\cdot 2^j} (2^{j+2}-1)H^{(j+2)}_{\frac{p-1}{2}}p^j&=0.\end{align}
\ 
\begin{proof}[Proof of Theorem \ref{prop4}] We have
\begin{align*} H^{(k)}_{p-1}&=\sum_{i=1}^{p-1}\frac{1}{i^k}=\sum_{1\le i \le p-1}\frac{1}{(p-i)^k}=(-1)^k\sum_{i=1}^{p-1}\frac{1}{i^k(1-\frac{p}{i})^k}\\
&=(-1)^k\sum_{i=1}^{p-1}\sum_{j\ge 0}\binom {j+k-1}{j}\frac{p^j}{i^{k+j}}    \text {     \ \ \      from \eqref{e2a}}\\
&=(-1)^k\sum_{j\ge 0}p^j\sum_{i=1}^{p-1}\binom {j+k-1}{j}\frac{1}{i^{k+j}}\\
&=(-1)^k \sum_{j\ge0} \binom {j+k-1}{j}H^{(k+j)}_{p-1}p^j.\\
H^{(2k)}_{p-1}&=\sum_{i=1}^{p-1}\frac{1}{i^{2k}}=\sum_{i=1}^\frac{p-1}{2}\frac{1}{i^{2k}}+\sum_{i=1}^\frac{p-1}{2}\frac{1}{(p-i)^{2k}}=\sum_{i=1}^\frac{p-1}{2}\frac{1}{i^{2k}}+\sum_{i=1}^\frac{p-1}{2}\frac{1}{i^{2k}(1-\frac{p}{i})^{2k}}\\
&=2H^{(2k)}_{\frac{p-1}{2}}+ \sum_{i=1}^\frac{p-1}{2}\sum_{j\ge 1}\binom {j+2k-1}{j}\frac{p^j}{i^{2k+j}}    \text {     \ \ \      from \eqref{e2a}}\\
&=2H^{(2k)}_{\frac{p-1}{2}}+ \sum_{j\ge 1} \binom {j+2k-1}{j}H^{(2k+j)}_{\frac{p-1}{2}}p^j.\\
H^{(k)}_{\frac{p-1}{2}}&=\sum_{i=1}^\frac{p-1}{2}\frac{1}{i^k}=\sum_{i=1}^\frac{p-1}{2}\frac{1}{(\frac{p+1}{2}-i)^k}=(-2)^k\sum_{i=1}^\frac{p-1}{2}\frac{1}{(2i-1)^k(1-\frac{p}{2i-1})^k}\\
&=(-2)^k\sum_{i=1}^\frac{p-1}{2}\sum_{j\ge 0}\binom {j+k-1}{j}\frac{p^j}{(2i-1)^{k+j}}    \text {     \ \ \      from \eqref{e2a}}\\
&=(-2)^k\sum_{j\ge 0}p^j\binom {j+k-1}{j}\sum_{i=1}^\frac{p-1}{2}\frac{1}{(2i-1)^{k+j}}\\
&=(-2)^k\sum_{j\ge 0}\binom {j+k-1}{j}\left(\sum_{i=1}^{p-1}\frac{1}{i^{k+j}}- \sum_{i=1}^\frac{p-1}{2}\frac{1}{(2i)^{k+j}}\right)p^j \\
&=(-2)^k \sum_{j\ge0} \binom {j+k-1}{j}\left(H^{(k+j)}_{p-1}-\frac{1}{2^{k+j}}H^{(k+j)}_{\frac{p-1}{2}}\right)p^j  \\
&=2^k H^{(k)}_{p-1}- (-1)^k \sum_{j\ge0} \binom {j+k-1}{j}\frac{1}{2^{j}}H^{(k+j)}_{\frac{p-1}{2}}p^j \text{ by using \eqref{e10ee}}\\
&=2^k H^{(k)}_{p-1}- (-1)^k H^{(k)}_{\frac{p-1}{2}}+\sum_{j\ge1} \binom {j+k-1}{j}\frac{1}{2^{j}}H^{(k+j)}_{\frac{p-1}{2}}p^j.
\end{align*}
After rearragement, this gives  \eqref{e10eee} and then \eqref{e10eeeff} is obtained by comparing \eqref{e10eed} to \eqref{e10eee} in the even case, and eliminating $H^{(2k)}_{p-1}$.
\end{proof}

\  \\


\begin{thm}\label{prop3}    
Let  $ k \ge 1$ a natural integer, and $p$ a prime. When $p \ge 2k+3$, the following congruences hold: 
\begin{align}\label{e8bbf}  H^{(2k)}_{p - 1}&\equiv p\frac{2k}{2k+1}B_{p-1-2k}\pmod {p^2},\\
\label{e10eeeb1} H^{(2k)}_{\frac{p-1}{2}}&\equiv p\frac{k(2^{2k+1} -1)}{2k+1}B_{p-1-2k} \pmod{p^2},\\
\label{e9bb} H^{(2k-1)}_{p-1}&\equiv-p^2\frac{k(2k-1)}{2k+1}B_{p-1-2k} \pmod{p^3},\\
\label{e9bbs} H^{(2k+1)}_{\frac{p-1}{2}}&\equiv  \frac{2(1-4^k)}{2k+1}B_{p-1-2k}\pmod {p}.
\end{align}
Moreover, the last one is also valid when $p = 2k+1$. 
\end{thm}
	
\begin{proof}\label{} All this is already well-known: see for example \cite{Glaisher1900}, \cite{Lehmer38}, \cite{Bayat97},  \cite{Sun00}. From Fermat little theorem, for $1\le i \le p-1$, we have $\frac{1}{i^{k}}-\frac{i^{p-1}}{i^{k}} \equiv 0 \pmod p$,
hence, by squaring, we have $(\frac{1}{i^{k}}-i^{p-1-k})^2 \equiv 0 \pmod {p^2}$, that is $\frac{1}{i^{2k}} \equiv 2i^{p-1-2k}-i^{2p-2-2k} \pmod {p^2}$ and then
	\begin{align*}H^{(2k)}_{p - 1}& \equiv 2\sum_{i=1}^{ p - 1}i^{p-1-2k}-\sum_{i=1}^{ p - 1}i^{2p-2-2k}\pmod {p^2},\end{align*}
but from \eqref{e2d} we have  $$\sum_{i=1}^{p}i^{h}\equiv p(-1)^hB_h +p^2(-1)^{h-1} \frac{h}{2}B_{h-1}\pmod {p^3},$$ that is, if $h \ge3$,
$$\sum_{i=1}^{ p - 1}i^{h}\equiv p(-1)^hB_h +p^2(-1)^{h-1} \frac{h}{2}B_{h-1}\pmod {p^3}.$$ 
Hence, as the odd-index Bernoulli numbers are zero, we have, for $h \ge2$,
$$\sum_{1\le i \le p-1}i^{2h}\equiv pB_{2h} \pmod {p^3}.$$  
And then
\begin{align*}H^{(2k)}_{p - 1}& \equiv 2\sum_{i=1}^{ p - 1}i^{p-1-2k}-\sum_{i=1}^{ p - 1}i^{2p-2-2k}\equiv p(2B_{p-1-2k}-B_{2p-2-2k})\pmod {p^2}.\end{align*}
But, since $p\ge 2k+3$, $p-1-2k$ is not divisible by $p-1$ then, from Kummer congruence \eqref{e2ez}, we see that $B_{2p-2-2k} \equiv  \frac{2k+2}{2k+1}B_{p-1-2k} \bmod p$  and therefore
$$H^{(2k)}_{p - 1}\equiv p\frac{2k}{2k+1}B_{p-1-2k}\pmod {p^2},$$ which is \eqref{e8bbf}. Then from \eqref{e10eeeb}, we have
$$ H^{(2k)}_{\frac{p-1}{2}}\equiv p\frac{k(2^{2k+1} -1)}{2k+1}B_{p-1-2k} \pmod{p^2},$$ which is \eqref{e10eeeb1}. We also have
\begin{align*}H^{(2k-1)}_{p-1}&=\sum_{i=1}^{p-1} \frac{1}{i^{2k-1}}=\sum_{i=1}^{p-1} \frac{1}{(p-i)^{2k-1}}=-\sum_{i=1}^{p-1} \frac{1}{i^{2k-1}(1-\frac{p}{i})^{2k-1}}\\
&\equiv -H^{(2k-1)}_{p-1}-p(2k-1)H^{(2k)}_{p-1}-p^2\frac{(2k)(2k-1)}{2}H^{(2k+1)}_{p-1}  \pmod{p^3}.\end{align*}
Then, rearranging and accounting for \eqref{e8bbf}, we obtain \eqref{e9bb} as 
 \begin{align*}
H^{(2k-1)}_{p -1}&\equiv-p\frac{2k-1}{2}H^{(2k)}_{p-1} \pmod{p^3}\\
&\equiv-p^2\frac{k(2k-1)}{2k+1}B_{p-1-2k} \pmod{p^3}. \end{align*}
For the case $p\ge2k+3$, \eqref{e9bbs} is obtained starting from \eqref{e10eeea} and taking \eqref{e8bbf} and \eqref{e10eeeb1} into account
\begin{align*}H^{(2k)}_{p-1}- 2H^{(2k)}_{\frac{p-1}{2}}&\equiv p\cdot 2kH^{(2k+1)}_{\frac{p-1}{2}} \pmod {p^2},\\
p\frac{2k}{2k+1}B_{p-1-2k}-p\frac{2k(2^{2k+1} -1)}{2k+1}B_{p-1-2k}&\equiv p\cdot 2kH^{(2k+1)}_{\frac{p-1}{2}} \pmod {p^2},\\
\frac{2(1-4^k)}{2k+1}B_{p-1-2k}&\equiv H^{(2k+1)}_{\frac{p-1}{2}} \pmod {p}.
\end{align*}
Finally, for the case $p=2k+1$, \eqref{e9bbs} is obtained by Fermat little theorem, and Eisenstein congruence \eqref{1.3} as  $$H^{(p)}_{\frac{p-1}{2}}\equiv H_{\frac{p-1}{2}}\equiv -2\cdot q_p \equiv \frac{2(1-4^{\frac{p-1}{2}})}{p}B_0\pmod {p}.$$
\end{proof}

\begin{col} For $p \ge 2k+3$, we have
\begin{align}
	\label{e7}    H^{(2k)}_{\frac{p-1}{2}} &\equiv 0 \pmod p, \\
\label{e8} H^{(2k)}_{p-1} &\equiv 0 \pmod p, \\
	\label{e9} H^{(2k-1)}_{p-1}  &\equiv 0 \pmod {p^2}.  \end{align}
	Moreover, when $B_{p-2k-1}\equiv 0 \pmod p$, in other words when $(p, p-2k-1)$ is an irregular pair, we have
	\begin{align}\label{e8b}   H^{(2k)}_{p-1} &\equiv 0 \pmod {p^2}, \\
\label{e8bb}   H^{(2k)}_{\frac{p-1}{2}} &\equiv 0 \pmod {p^2}, \\
\label{e9b}   H^{(2k-1)}_{p-1}  &\equiv 0 \pmod {p^3},  \\
	\label{e9bbb}   H^{(2k+1)}_{\frac{p-1}{2}}  &\equiv 0 \pmod {p}.  \end{align}
\end{col}
\ \\

\noindent We will now derive congruences modulo arbitrary prime powers for the Harmonic numbers, involving the Bernoulli numbers. The next theorem is essentially already known: it was originally obtained from quite advanced mathematics, involving $p$-adic L-functions  (\cite{Washington98}, Theorem 1). Our derivation will be more elementary.

\begin{thm}\label{3.1}  Let $n$, $i$ be natural integers, $p$ a prime. The congruence
\begin{equation}\label{ee10bis}\sum_{j=0}^{2n+1}\binom{j+2i}{2i}B_j H^{(j+2i+1)}_{p-1}(-p)^j\equiv 0 \pmod {p^{2n+m}} \end{equation} 
holds with the following values of $m$ under the corresponding restrictions on $p$: 
\begin{align*} m&=1 \text{ when } p\ge 2, \\
m&=2 \text{ when } p\ge 3, \\m&=3 \text{ when } p\ge 5 \text{ or }  p=3  \text{ and }n\equiv 0 \pmod 3, \\m&=4 \text{ when } p\ge 2n+2i+7, \\m&=5 \text{ when }  (p, p-2n-2i-5) \text{ is an irregular pair.} \end{align*}
\end{thm}

\noindent \text{\bf{Remark.}} Congruence \eqref{1.2} is a particular case, $n=i=0$,  $m=4$,  of \eqref{ee10bis}. Similar congruences are 
	\begin{align*}  H_{p-1} + \frac{p}{2}H^{(2)}_{p-1} + \frac{p^2}{6}H^{(3)}_{p-1} & \equiv 0 \pmod{p^6} \text{ \ \ \ when } p\ge 9,\\
	  H_{p-1} + \frac{p}{2}H^{(2)}_{p-1} + \frac{p^2}{6}H^{(3)}_{p-1} - \frac{p^4}{30}H^{(5)}_{p-1} &\equiv 0 \pmod{p^8} \text{ \ \ \ when } p\ge 11,\\
 &\text{etc.}\end{align*}
\noindent \text{\bf{Remark.}} Also in particular when $n=0$, we have
\begin{equation}\label{zzzz} H^{(2i+1)}_{p-1}+p\frac{2i+1}{2}H^{(2i+2)}_{p-1}\equiv 0 \pmod {p^{m}}, \end{equation}
with: 
\begin{align*} m&=1 \text{ when } p\ge 2, \\
m&=2 \text{ when } p\ge 3, \\m&=3 \text{ when } p\ge 5 \text{ or }  p=3  \text{ and }n\equiv 0 \pmod 3, \\m&=4 \text{ when } p\ge 2i+7, \\m&=5 \text{ when }  (p, p-2i-5) \text{ is an irregular pair.} \end{align*}

\noindent \text{\bf{Example.}} When $m=5$, since $(37,32)$ is an irregular pair, we have $i=0$, and
	$$\sum_{j=1}^{36} \frac{1}{j} +\frac{37}{2} \sum_{j=1}^{36} \frac{1}{j^2} =\frac{1422091936194747472864459922257}{41704772176589465865841920000}=\frac{N}{D}$$ 
	and $N= 37^5\cdot1123\cdot9133\cdot1999520400972139$ is divisible by $37^5$, as expected.\\

\begin{col}\label{Remark4}
 For any odd prime  $p$, we also have the following weaker $p$-adically converging series involving generalized Harmonic numbers:
\begin{equation}\label{ee10biss}\sum_{j=0 }^{k-1}\binom{j+2i}{2i}B_j H^{(j+2i+1)}_{p-1}(-p)^j \equiv 0 \pmod {p^k}.\end{equation}
  In particular, when $i=0$, we have
	\begin{equation}\label{zzzza}\sum_{0\le j}B_j H^{(j+1)}_{p-1}(-p)^j=0.  \end{equation}
	\end{col}
\noindent \text{\bf{Remark.}} Note that \eqref{zzzza} coincide with \eqref{e10eex} only up to $j \le 1$.\\
\ 
\begin{proof}[Proof of Theorem \ref{3.1}]
We first show how \eqref{ee10biss} is obtained from \eqref{ee10bis}. Suppose $p\ge3$. In the case $k=1$, \eqref{ee10biss} reduces to \eqref{e10eeez}. Suppose now $k>1$. If $k$ is even, then $k=2n+2$ ($n \ge 0$) and by  \eqref{ee10bis}, we have
\begin{align*}\sum_{j=0 }^{k-1}\binom{j+2i}{2i}B_j H^{(j+2i+1)}_{p-1}(-p)^j &=\sum_{j=0 }^{2n+1}\binom{j+2i}{2i}B_j H^{(j+2i+1)}_{p-1}(-p)^j \\
&\equiv 0 \pmod {p^{2n+2} = p^k}.
\end{align*}
 If $k$ is odd, then $k=2n+3$ ($n \ge 0$) and  $B_{2n+3}=0$, then
\begin{align*}\sum_{j=0 }^{k-1}\binom{j+2i}{2i}B_j H^{(j+2i+1)}_{p-1}(-p)^j &=\sum_{j=0 }^{2n+2}\binom{j+2i}{2i}B_j H^{(j+2i+1)}_{p-1}(-p)^j  \\
&=\sum_{j=0 }^{2n+3}\binom{j+2i}{2i}B_j H^{(j+2i+1)}_{p-1}(-p)^j \\
& \equiv 0 \pmod {p^{2n+4}}  \text { (by  \eqref{ee10bis}) }\\
& \equiv 0 \pmod {p^{2n+3}= p^k}.
\end{align*}

\noindent Now, for the proof of \eqref{ee10bis}, let $S$ be the left hand side in \eqref{ee10bis}. From \eqref{e10ee} we have
\begin{align*}H^{(j+2i+1)}_{p-1}&=(-1)^{j+2i+1} \sum_{h\ge0} \binom {h+j+2i}{h}H^{(j+2i+h+1)}_{p-1}p^h\end{align*}  
then
\begin{align*}S&=-\sum_{j=0 }^{2n+1}\binom{j+2i}{2i}B_j \sum_{h\ge0} \binom {h+j+2i}{h}H^{(j+2i+h+1)}_{p-1}p^{h+j}. \end{align*}
Let $k=h+j$,  $j\le 2n+1$ and $j\le k$, hence $j\le \min(2n+1,k)$ and we can derive the following series expansion of $S$ in terms of $p^k$:
\begin{align*}
S&=-\sum_{k\ge0}p^{k}H^{(k+2i+1)}_{p-1}\sum_{j=0 }^{\min(2n+1,k)}B_j\binom{j+2i}{2i}\binom {k+2i}{j+2i}   \\
&=-\sum_{k\ge0}p^{k}H^{(k+2i+1)}_{p-1}\binom {k+2i}{2i}\sum_{j=0 }^{\min(2n+1,k)}B_j\binom {k}{j}  \\
&=-\sum_{k=0}^{2n+1}p^{k}H^{(k+2i+1)}_{p-1}\binom {k+2i}{2i}\sum_{j=0}^kB_j\binom {k}{j}\\
&-\sum_{k\ge2n+2}p^{k}H^{(k+2i+1)}_{p-1}\binom {k+2i}{2i}\sum_{j=0}^{2n+1}B_j\binom {k}{j}\\
&=-\sum_{k=0}^{2n+1}p^{k}H^{(k+2i+1)}_{p-1}\binom {k+2i}{2i}(-1)^kB_k\\
&-\sum_{k\ge2n+2}p^{k}H^{(k+2i+1)}_{p-1}\binom {k+2i}{2i}\sum_{j=0}^{2n+1}B_j\binom {k}{j} \\
&=-S -\sum_{k\ge2n+2}p^{k}H^{(k+2i+1)}_{p-1}\binom {k+2i}{2i}\sum_{j=0}^{2n+1}B_j\binom {k}{j}.
\end{align*}
Hence
\begin{align*}2S&=-\sum_{k\ge2n+2}p^{k}H^{(k+2i+1)}_{p-1}\binom {k+2i}{2i}\sum_{j=0}^{2n+1}B_j\binom {k}{j}.\end{align*}
Note that, in the above derivation, it is made use of o\eqref{e2b} and also of the identity $\binom{j+2i}{2i}\binom {k+2i}{j+2i}=\binom {k+2i}{2i}\binom {k}{j}$. Moreover, from \eqref{e2b}, we also have $\sum_{j=0}^{2n+1}B_j\binom{2n+2}{j}=0$.
Hence
\begin{equation}\label{e10cxbis} 2S=-\sum_{k\ge2n+3}p^{k}H^{(k+2i+1)}_{p-1}\binom {k+2i}{2i}\sum_{j=0}^{2n+1}B_j\binom {k}{j}.\end{equation}

\noindent By the Von Staudt-Clausen theorem \eqref{e2e}, we know that $p$ may divide the denominator of any $B_j$ only once at most, then we must have $$2S \equiv 0 \pmod{p^{2n+2}} $$ and this makes a proof of the above claim for the cases $m=1$ and $m=2$. Also, again from \eqref{e2b}, $\sum_{j=0}^{2n+1}B_j\binom{2n+3}{j}=-(2n+3)B_{2n+2}$, then \eqref{e10cxbis} becomes 
\begin{equation}\label{e10cybis} \begin{split}2S&=p^{2n+3}H^{(2n+2i+4)}_{p-1}\binom {2n+2i+3}{2i}(2n+3)B_{2n+2} \\
&-\sum_{k\ge 2n+4}p^{k}H^{(k+2i+1)}_{p-1}\binom {k+2i}{2i}\sum_{j=0}^{2n+1}B_j\binom {k}{j}.\end{split}\end{equation}

\noindent The second term on the right hand side of \eqref{e10cybis} is unconditionnally zero modulo $p^{2n+3}$, by Von Staudt-Clausen. The first term on the right hand side of \eqref{e10cybis} is also zero modulo $p^{2n+3}$ when $2n+2$ is not a multiple of $p-1$, also by Von Staudt-Clausen. Then the congruence modulo $p^{2n+3}$ might fail only when $p=3$ or when $n+1$ is a multiple of $\frac{p-1}{2}$. But in the latter case, we have
\begin{align*} H^{(2n+4)}_{p-1}=\sum_{i=1}^{p-1}\frac{1}{i^2 i^{2(n+1)}} \equiv \sum_{i=1}^{p-1}\frac{1}{i^2} &=H^{(2)}_{p-1}\\
&\equiv 0\pmod p \text{ \ when $p\ge 5$, \  by \eqref{e8}} \end{align*}  
and this makes a compensation. Note that in the case where $p=3$ there is also a compensation when $3$ divides $2n+3$, that is when $n$ is a multiple of $3$. This remark completes the proof in the case $m=3$.
\\Now, if $p\ge 2n+2i+7$, from Theorem \ref{prop3} \eqref{e8}, $H^{(2n+2i+4)}_{p-1}\equiv 0 \pmod p$ and, again from the Von Staudt-Clausen theorem, $p$ is large enough so that it never divides the denominators of $B_j$ when $0\le j\le 2n+1$ then when $p\ge 2n+2i+7$, we have $ 2S \equiv 0 \pmod {p^{2n+4}}$, this completes the proof for the case $m=4$.\\
Finally, if  $B_{p-2n-2i-5}\equiv 0 \pmod p$, then $H^{(2n+2i+4)}_{p-1}\equiv 0 \pmod {p^2}$, from Theorem \ref{prop3}, and since as soon as $p\ge3$,  $H^{(2n+2i+5)}_{p-1}\equiv 0 \pmod {p}$, by Theorem \ref{prop4} and this completes the proof for the case $m=5$.
\end {proof}
\noindent Now, we give a $p$-adic expansion for $H^{(2i)}_\frac{p-1}{2}$ which seems to be new. It is in the same spirit as the previous one, though.
\begin{thm}\label{H2kp-1/2}Let $p$ be an odd prime, $n, i$ positive integers and $j$ a non-negative integer. Let $C_j:=2B_{j+2}+(-1)^jB_{j+1}+\frac{1}{2}$. The congruence
\begin{equation}\label{eecj}\sum_{j=0}^{2n-1}\binom{j+2i-1}{j+1}\frac{2^{j+2i}-1}{2^j} C_jH^{(j+2i)}_\frac{p-1}{2}p^j\equiv 0 \pmod {p^{2n+m}} \end{equation} 
holds with the following value for $m$ and the corresponding restrictions on $p$: we have $m=0$ when $p\ge 3$, $m=1$ when $p>2n+1$, and  $m=2$ when
$$\binom{2n+2i}{2n+2}\left(2^{2n+2i+1}-1\right) H^{(2n+2i+1)}_\frac{p-1}{2}\left( (2n+3)B_{2n+2}+\frac{n}{2}\right)\equiv 0 \pmod p. $$

\begin{proof}
We start from a reformulation of \eqref{e10eeeff}, with $j \ge 1$, $g \ge 0$. We have
\begin{equation} \label {abc}\begin{split}
H^{(2j+2g)}_{\frac{p-1}{2}}=&-\sum_{h\ge0} \binom {2h+2j+2g-1}{2h}\frac{2^{2h+2g+2j}-1}{ 2^{2j+2g}-1}H^{(2h+2g+2j)}_{\frac{p-1}{2}}\left(\frac{p}{2}\right)^{2h}\\
&-\sum_{h\ge0} \binom {2h+2j+2g}{2h+1}\frac{2^{2h+2g+2j+1}-1}{2^{2j+2g}-1}H^{(2h+2g+2j+1)}_{\frac{p-1}{2}}\left(\frac{p}{2}\right)^{2h+1}.
\end{split}\end{equation}

\noindent Let $$ S=\sum_{j=1}^{n} \binom {2j+2g-1}{2g}B_{2j}(2^{2g+2j}-1)H^{(2g+2j)}_{\frac{p-1}{2}}\left(\frac{p}{2}\right)^{2j}.$$ 

\noindent Substituting  $H^{(2j+2g)}_{\frac{p-1}{2}}$ from \eqref{abc} in the above expression for $S$, we obtain\\
\noindent $S=-\sum_{j=1}^{n} \binom {2j+2g-1}{2g}B_{2j}\sum_{h\ge0}\binom {2h+2j+2g-1}{2h}(2^{2h+2g+2j}-1)H^{(2h+2g+2j)}_{\frac{p-1}{2}}\left(\frac{p}{2}\right)^{2j+2h}\\
-\sum_{j=1}^{n} \binom {2j+2g-1}{2g}B_{2j}\sum_{h\ge0} \binom {2h+2j+2g}{2h+1}(2^{2h+2g+2j+1}-1)H^{(2h+2g+2j+1)}_{\frac{p-1}{2}}\left(\frac{p}{2}\right)^{2j+2h+1}$.\\

\noindent Now, we let $2k=2h+2j$ and we will replace the summation index $h$ by $k$. The new index $k$ runs from $1$, with no upper bound and the index $j$ runs from $1$ to $\min(n,k)$, since $1\le j\le n$ and $2j=2k-2h \le 2k$. Then, with the new summation indices and after inverting the sums, we have
\begin{align*}&S=-\sum_{1\le k}(2^{2g+2k}-1)H^{(2g+2k)}_{\frac{p-1}{2}} \left(\frac{p}{2}\right)^{2k}\sum_{j=1}^{\min(n,k)}\binom {2j+2g-1}{2g}\binom {2g+2k-1}{2j+2g-1}B_{2j} \\
&-\sum_{1\le k}(2^{2g+2k+1}-1)H^{(2g+2k+1)}_{\frac{p-1}{2}} \left(\frac{p}{2}\right)^{2k+1}\sum_{j=1}^{\min(n,k)}\binom {2j+2g-1}{2g}\binom {2g+2k}{2j+2g-1}B_{2j}.
\end{align*} 
\noindent But 
\begin{align*}\binom {2j+2g-1}{2g}\binom {2g+2k-1}{2j+2g-1}&=\binom {2k+2g-1}{2g}\binom {2k-1}{2j-1}\end{align*} and 
\begin{align*}\binom {2j+2g-1}{2g}\binom {2g+2k}{2j+2g-1}&=\binom {2k+2g}{2g}\binom {2k}{2j-1}. \end{align*} 
Then
\begin{align*}S=&-\sum_{1\le k}\binom {2k+2g-1}{2g}(2^{2g+2k}-1)H^{(2g+2k)}_{\frac{p-1}{2}} \left(\frac{p}{2}\right)^{2k}\sum_{j=1}^{\min(n,k)}\binom {2k-1}{2j-1}B_{2j} \\
&-\sum_{1\le k}\binom {2k+2g}{2g}(2^{2g+2k+1}-1)H^{(2g+2k+1)}_{\frac{p-1}{2}} \left(\frac{p}{2}\right)^{2k+1}\sum_{j=1}^{\min(n,k)}\binom {2k}{2j-1}B_{2j}.\end{align*}
\begin{align*} S=&-\sum_{k=1}^{ n}\binom {2k+2g-1}{2g}(2^{2g+2k}-1)H^{(2g+2k)}_{\frac{p-1}{2}} \left(\frac{p}{2}\right)^{2k}\sum_{j=1}^{ k}\binom {2k-1}{2j-1}B_{2j} \\
&-\sum_{k=1}^{ n}\binom {2k+2g}{2g}(2^{2g+2k+1}-1)H^{(2g+2k+1)}_{\frac{p-1}{2}} \left(\frac{p}{2}\right)^{2k+1}\sum_{j=1}^{ k}\binom {2k}{2j-1}B_{2j} \\
&-\binom {2n+2g+1}{2g}(2^{2g+2n+2}-1)H^{(2g+2n+2)}_{\frac{p-1}{2}} \left(\frac{p}{2}\right)^{2n+2}\sum_{j=1}^{ n}\binom {2n+1}{2j-1}B_{2j} \\
&-\binom {2n+2g+2}{2g}(2^{2g+2n+3}-1)H^{(2g+2n+3)}_{\frac{p-1}{2}} \left(\frac{p}{2}\right)^{2n+3}\sum_{j=1}^{ n}\binom {2n+2}{2j-1}B_{2j} \\
&-\sum_{n+2\le k }\binom {2k+2g-1}{2g}(2^{2g+2k}-1)H^{(2g+2k)}_{\frac{p-1}{2}} \left(\frac{p}{2}\right)^{2k}\sum_{j=1}^{ n}\binom {2k-1}{2j-1}B_{2j} \\
&-\sum_{n+2\le k}\binom {2k+2g}{2g}(2^{2g+2k+1}-1)H^{(2g+2k+1)}_{\frac{p-1}{2}} \left(\frac{p}{2}\right)^{2k+1}\sum_{j=1}^{ n}\binom {2k}{2j-1}B_{2j}.
\end{align*}

\noindent We now make use of Lemma \ref{lem1}, so that
\begin{align*}S=&-\sum_{k=1}^{ n}\binom {2k+2g-1}{2g}(2^{2g+2k}-1)H^{(2g+2k)}_{\frac{p-1}{2}} \left(\frac{p}{2}\right)^{2k}\left(B_{2k}+B_{2k-1}+\frac{1}{2}\right) \\
&-\sum_{k=1}^{ n}\binom {2k+2g}{2g}(2^{2g+2k+1}-1)H^{(2g+2k+1)}_{\frac{p-1}{2}} \left(\frac{p}{2}\right)^{2k+1}\left(\frac{1}{2}-B_{2k}\right) \\
&-\binom {2n+2g+1}{2g}(2^{2g+2n+2}-1)H^{(2g+2n+2)}_{\frac{p-1}{2}} \left(\frac{p}{2}\right)^{2n+2}\frac{1}{2} \\
&-\binom {2n+2g+2}{2g}(2^{2g+2n+3}-1)H^{(2g+2n+3)}_{\frac{p-1}{2}} \left(\frac{p}{2}\right)^{2n+3}\left(\frac{1}{2}-(2n+3)B_{2n+2}\right)\\
&-\sum_{n+2\le k }\binom {2k+2g-1}{2g}(2^{2g+2k}-1)H^{(2g+2k)}_{\frac{p-1}{2}} \left(\frac{p}{2}\right)^{2k}\sum_{j=1}^{ n}\binom {2k-1}{2j-1}B_{2j} \\
&-\sum_{n+2\le k}\binom {2k+2g}{2g}(2^{2g+2k+1}-1)H^{(2g+2k+1)}_{\frac{p-1}{2}} \left(\frac{p}{2}\right)^{2k+1}\sum_{j=1}^{ n}\binom {2k}{2j-1}B_{2j}.
\end{align*}

\noindent By the Von Staudt Clausen theorem, $p$ may divide the inner sum in the two last lines only once at most, and then the last line in the above expression for $S$ is $0 \bmod p^{2n+4}$, and, since by \eqref{e10eeee} $(2^{2g+2k}-1)H^{(2g+2k)}_{\frac{p-1}{2}} \equiv 0 \bmod p$, the fifth line is also $0 \bmod p^{2n+4}$. Then, we obtain\\
\begin{align*}S\equiv &-\sum_{k=1}^{ n}\binom {2k+2g-1}{2g}(2^{2g+2k}-1)H^{(2g+2k)}_{\frac{p-1}{2}} \left(\frac{p}{2}\right)^{2k}\left(B_{2k}+B_{2k-1}+\frac{1}{2}\right) \\
&-\sum_{k=1}^{ n}\binom {2k+2g}{2g}(2^{2g+2k+1}-1)H^{(2g+2k+1)}_{\frac{p-1}{2}} \left(\frac{p}{2}\right)^{2k+1}\left(\frac{1}{2}-B_{2k}\right) \\
&-\binom {2n+2g+1}{2g}(2^{2g+2n+2}-1)H^{(2g+2n+2)}_{\frac{p-1}{2}} \left(\frac{p}{2}\right)^{2n+2}\frac{1}{2} \\
&-\binom {2n+2g+2}{2g}(2^{2g+2n+3}-1)H^{(2g+2n+3)}_{\frac{p-1}{2}} \left(\frac{p}{2}\right)^{2n+3}\left(\frac{1}{2}-(2n+3)B_{2n+2}\right) \\
&\pmod{p^{2n+4}}.
\end{align*}
Now by \eqref{e10eeee}, we have
$$(2^{2g+2n+2}-1)H^{(2g+2n+2)}_{\frac{p-1}{2}} \equiv -p\frac{n+g+1}{2}(2^{2g+2n+3}-1)H^{(2g+2n+3)}_{\frac{p-1}{2}} \bmod p^2,$$
and after some calculations the third and forth lines in the above expression for $S$ can be merged so that:
\begin{align*}S\equiv &-\sum_{k=1}^{ n}\binom {2k+2g-1}{2g}(2^{2g+2k}-1)H^{(2g+2k)}_{\frac{p-1}{2}} \left(\frac{p}{2}\right)^{2k}\left(B_{2k}+B_{2k-1}+\frac{1}{2}\right) \\
&-\sum_{k=1}^{ n}\binom {2k+2g}{2g}(2^{2g+2k+1}-1)H^{(2g+2k+1)}_{\frac{p-1}{2}} \left(\frac{p}{2}\right)^{2k+1}\left(\frac{1}{2}-B_{2k}\right) \\
&+\binom {2n+2g+2}{2g}(2^{2g+2n+3}-1)H^{(2g+2n+3)}_{\frac{p-1}{2}} \left(\frac{p}{2}\right)^{2n+3}\left(\frac{n}{2}+(2n+3)B_{2n+2}\right)\\
&\pmod{p^{2n+4}}.
\end{align*}

\noindent Recall that $ S= \sum_{k=1}^{ n}\binom {2k+2g-1}{2g}(2^{2g+2k}-1)H^{(2g+2k)}_{\frac{p-1}{2}} \left(\frac{p}{2}\right)^{2k}B_{2k}$, 
then after rearranging and dividing throughout by $p^2$, we obtain the following congruence$\pmod  {p^{2n+2}}$:

\begin{align*}0\equiv &-\sum_{k=1}^{ n}\binom {2k+2g-1}{2g}(2^{2g+2k}-1)H^{(2g+2k)}_{\frac{p-1}{2}} \left(\frac{p}{2}\right)^{2k-2}\left(2B_{2k}+B_{2k-1}+\frac{1}{2}\right) \\
&-\sum_{k=1}^{ n}\binom {2k+2g}{2g}(2^{2g+2k+1}-1)H^{(2g+2k+1)}_{\frac{p-1}{2}} \left(\frac{p}{2}\right)^{2k-1}\left(\frac{1}{2}-B_{2k}\right) \\
&+\binom {2n+2g+2}{2g}(2^{2g+2n+3}-1)H^{(2g+2n+3)}_{\frac{p-1}{2}} \left(\frac{p}{2}\right)^{2n+1}\left(\frac{n}{2}+(2n+3)B_{2n+2}\right)\\
&\pmod{p^{2n+2}} .\end{align*}
That is
\begin{align*}0\equiv &-\sum_{\underset{k \text{even}}{k=0}}^{2n-2}\binom {k+2g+1}{2g}(2^{2g+k+2}-1)H^{(2g+k+2)}_{\frac{p-1}{2}} \left(\frac{p}{2}\right)^{k}\left(2B_{k+2}+B_{k+1}+\frac{1}{2}\right) \\
&-\sum_{\underset{k \text{odd}}{k=1}}^{2n-1}\binom {k+2g+1}{2g}(2^{2g+k+2}-1)H^{(2g+k+2)}_{\frac{p-1}{2}} \left(\frac{p}{2}\right)^{k}\left(\frac{1}{2}-B_{k+1}\right) \\
&+\binom {2n+2g+2}{2g}(2^{2g+2n+3}-1)H^{(2g+2n+3)}_{\frac{p-1}{2}} \left(\frac{p}{2}\right)^{2n+1}\left(\frac{n}{2}+(2n+3)B_{2n+2}\right)\\
&\pmod{p^{2n+2}}.
\end{align*}

\noindent But again, since from $B_3$, the odd-indexed Bernoulli number are $0$, we have\\ 
\begin{align*}0\equiv &-\sum_{k=0}^{ 2n-2}\binom {k+2g+1}{2g}(2^{2g+k+2}-1)H^{(2g+k+2)}_{\frac{p-1}{2}} \left(\frac{p}{2}\right)^{k}\left(2B_{k+2}+\frac{1}{2}\right) \\
&-\sum_{\underset{k \text{even}}{k=0}}^{2n-2}\binom {k+2g+1}{2g}(2^{2g+k+2}-1)H^{(2g+k+2)}_{\frac{p-1}{2}} \left(-\frac{p}{2}\right)^{k}B_{k+1} \\
&-\sum_{\underset{k \text{odd}}{k=1}}^{2n-1}\binom {k+2g+1}{2g}(2^{2g+k+2}-1)H^{(2g+k+2)}_{\frac{p-1}{2}} \left(\frac{p}{2}\right)^{k}(-1)^kB_{k+1} \\
&+\binom {2n+2g+2}{2g}(2^{2g+2n+3}-1)H^{(2g+2n+3)}_{\frac{p-1}{2}} \left(\frac{p}{2}\right)^{2n+1}\left(\frac{n}{2}+(2n+3)B_{2n+2}\right)\\
\equiv &\sum_{k=0}^{ 2n-2}\binom {k+2g+1}{2g}(2^{2g+k+2}-1)H^{(2g+k+2)}_{\frac{p-1}{2}} \left(\frac{p}{2}\right)^{k}\left(2B_{k+2}+(-1)^k B_{k+1}+\frac{1}{2}\right) \\
&-\binom {2n+2g+2}{2g}(2^{2g+2n+3}-1)H^{(2g+2n+3)}_{\frac{p-1}{2}} \left(\frac{p}{2}\right)^{2n+1}\left(\frac{n}{2}+(2n+3)B_{2n+2}\right) \\
&\pmod{p^{2n+2}}.
\end{align*}

\noindent Introducing $C_k$, and shifting $g$, so that $g\ge 1$, we obtain 
\begin{align*} 
 0\equiv &\sum_{k=0}^{2n-1}\binom {k+2g-1}{k+1}(2^{2g+k}-1)C_kH^{(2g+k)}_{\frac{p-1}{2}} \left(\frac{p}{2}\right)^{k}\\
&-\binom {2n+2g}{2n+2}(2^{2g+2n+1}-1)H^{(2g+2n+1)}_{\frac{p-1}{2}} \left(\frac{p}{2}\right)^{2n+1}\left(\frac{n}{2}+(2n+3)B_{2n+2}\right)\\
& \pmod{p^{2n+2}}.
\end{align*}

\noindent The latter congruence obviously  proves the theorem in the case $m=2$, but also in the case $m=1$ since when $p > 2n+1$, $ \frac{n}{2}+(2n+3)B_{2n+2}$ is $p$-integral: this is true for $p>2n+3$ by Von Staudt-Clausen, and also for $p=2n+3$ as $pB_{p-1}$ is also p-integral; and in the case $m=0$, under the only condition that $p$ be odd.
\end{proof}
\end {thm}
\begin{col}\label{Remark kk} When $i=1$, $m=0$, we have the following $p$-adically converging series, for $p$ an odd prime:  
\begin{equation}\label{eecjj}\sum_{0\le j}\left(2B_{j+2}+(-1)^jB_{j+1}+\frac{1}{2}\right)\left(2^{j+2}-1\right)H^{(j+2)}_\frac{p-1}{2}\left(\frac{p}{2}\right)^j=0. \end{equation} 
\end{col}

\noindent \text{\bf{Remark.}} Note that \eqref{eecjj} coincide with \eqref{e10eeeffx} up to $j\le1$ only.\\
\noindent \text{\bf{Remark.}} When $i=1$, $m=1$, Theorem \ref{H2kp-1/2} gives 
\begin{align*}H^{(2)}_\frac{p-1}{2} + \frac{7}{6}H^{(3)}_\frac{p-1}{2}p &\equiv 0 \pmod {p^{3}} \text { \ \ \ when  } p> 3\\   
H^{(2)}_\frac{p-1}{2} + \frac{7}{6}H^{(3)}_\frac{p-1}{2}p+\frac{13}{8}H^{(4)}_\frac{p-1}{2}p^2+\frac{31}{15}H^{(5)}_\frac{p-1}{2}p^3 &\equiv 0 \pmod {p^{5}} \text { \ \ \ when  } p> 5 \\
& \text{etc.} \end{align*}
\noindent \text{\bf{Remark.}} Note also how these congruences differ from what can be obtained from Theorem \ref{prop4} and its Corollary. They read differently from $j>1$ and they don't need the same restrictions on $p$. For example, \eqref{e10eeeff} limited at  order $ j \le 4$,  combined with \eqref{e10eeee} leads  to
\begin{align*}  
H^{(2)}_\frac{p-1}{2} + \frac{7}{6}H^{(3)}_\frac{p-1}{2}p+\frac{15}{8}H^{(4)}_\frac{p-1}{2}p^2+\frac{31}{12}H^{(5)}_\frac{p-1}{2}p^3 &\equiv 0 \pmod {p^{5}} \text { \ \ \ when  } p> 3.
 \end{align*}
\\ \
\begin{col}\label{Remark kkk} There are quite many cases where the congruence from Theorem \ref{H2kp-1/2} holds modulo $ p^{2n+2}$.
Most notably: \begin{align*} 
m&=2 \text{ when }i\ge2 \text{ and }p=2n+3, \\
m&=2 \text{ when } (p,p-2n-2i-1) \text{ is an irregular pair, } \\
m&=2 \text{ when }  p=2^{2n+2i+1}-1 \text{ ($p$ is a Mersenne prime), }\\
m&=2 \text{ when }  p=2n+2i+1 \text{ is a Wieferich prime ($p=1093, 3511,..$\cite{OEIS1}). }
 \end{align*}
\begin{proof}
The first case is clear because when $p=2n+3$, $(2n+3)B_{2n+2}+\frac{n}{2}$ is $p$-integral, and $2n+3$ divides $\binom{2n+2i}{2n+2}$ as soon as $i\ge 2$. The second case follows from \eqref{e9bbb}. The third case is clear because $2^{2n+2i+1}-2 > 2n+2$. The last case derives from \eqref{e9bbs} : $  H^{(p)}_{\frac{p-1}{2}}\equiv  \frac{2(1-2^{p-1})}{p}B_{0}\bmod {p}$ because by definition a Wieferich prime $p$ satisfies $2^{p-1}-1 \equiv 0 \bmod p^2$.
\end{proof}  
\end{col}

\noindent \text{\bf{Examples.}} (i) When $n=1=i$ and $p=37$, $(p,p-2i-2n-1)$ is an irregular pair, so $m=2$. We check: 
$$\sum_{j=1}^{18} \frac{1}{j^2} +37\frac{7}{6} \sum_{j=1}^{18} \frac{1}{j^3} =\frac{9356942544006649495921}	{175168974229337088000}$$
and $9356942544006649495921= 19\cdot37^4\cdot262768598968219$ is divisible by $37^4$.\\
(ii) When $n=1=i$ and $p=31=2^{2n+2i+1}-1$ so $m=2$ as well. We check that 
$$\sum_{j=1}^{15} \frac{1}{j^2} +31\frac{7}{6} \sum_{j=1}^{15}\frac{1}{j^3} =\frac{1804176116127398723}	{40110949726848000}$$
and $1804176116127398723= 19\cdot31^4\cdot619\cdot809\cdot3901153 $ is indeed divisible by $31^4$.\\
(iii) And at last, when $n=1$, $i=2$, $p=2n+3=5$, we also have $m=2$ as well. We check that
\begin{align*}\binom{3}{1}\frac{2^{4}-1}{2^0} C_0H^{(4)}_\frac{5-1}{2}5^0+\binom{4}{2}\frac{2^{5}-1}{2^1} C_1H^{(5)}_\frac{5-1}{2}5^1 &=15\left(1+\frac{1}{16}\right)+5\cdot31 \left(1+\frac{1}{32}\right)\\
&=\frac{3^2\cdot5^4}{2^5} \text{ \  is indeed divisible by $5^4$.} \end {align*}

\section{Extended congruences for  $H^{(2j-1)}_{\frac{p-1}{2}}$   ($j\ge 1$). }

The case of  $H^{(2j-1)}_{\frac{p-1}{2}}$   ($j\ge 1$) has not been adressed in Section 3. This is a more complicated case as it does not involve only Bernoulli numbers, but also the base-$2$ Fermat quotient. We begin with two propositions. 

\begin{prop}\label{4.1} Let $p$ be an odd prime and $n$  an integer, such that $p>\frac{n+1}{2}\ge 1$. We have the following congruence: 
\begin{equation}\label{4.1}\begin{split}H_{\frac{p-1}{2}}&+2\sum_{j=0}^{n-1}(-1)^j\frac{q_p^{j+1}}{j+1}p^j+2\delta^{n+1}_p q_pp^{n-1}\\
&+\sum_{1\le i <\frac{n+1}{2}}\frac{B_{p^{n-1}(p-1)-2i}}{2i}(2^{2i+1}-1)\left(\frac{p}{2}\right)^{2i} \equiv 0 \pmod {p^n}.\end{split}
\end{equation}
\end{prop}
\noindent \text{\bf{Remark.}} This proposition is very close to a particular case of Corollary 4.3 in \cite{Lin18} which was obtained from the theory of $p$-adic L-functions. Our statement here is stronger: at the small cost of the additional $2\delta^{n+1}_p q_pp^{n-1}$ term, we can be less restrictive on $p$.
\begin {proof}[Proof of Proposition \ref{4.1}] The proof starts from Faulhaber formula \eqref{e2d}. We have  
\begin{align*}\label{} \sum_{j=1 }^\frac{p-1}{2}j^i&=\frac{1}{i+1}\sum_{g=0}^i(-1)^g\binom{i+1}{g} B_g \left(\frac{p-1}{2}\right)^{i+1-g}\\
&=\frac{1}{i+1}\sum_{g=1}^{i+1}(-1)^{i+1-g}\binom{i+1}{g} B_{i+1-g} \left(\frac{p-1}{2}\right)^g\\
&=\frac{1}{i+1}\sum_{g=0}^i(-1)^{i-g} \frac{\binom{i+1}{g+1} B_{i-g} }{2^{g+1}}(p-1)^{g+1}\\
&=\frac{1}{i+1}\sum_{g=0}^i(-1)^{i-g} \frac{\binom{i+1}{g+1} B_{i-g} }{2^{g+1}}\sum_{m\ge 0}p^m (-1)^{g+1-m}\binom{g+1}{m}.
   \end{align*}
After inversion of the summations, with $\binom{i+1}{g+1}\binom{g+1}{m}=\binom{i+1}{m}\binom{i+1-m}{i-g}$ and some rearrangement, we obtain

\begin{align*}\label{} \sum_{j=1 }^\frac{p-1}{2}j^i=&\frac{(-1)^{i+1}}{(i+1)2^{i+1}} \sum_{0\le m}(-1)^m p^m\binom{i+1}{m}\sum_{g=0}^i 2^g\binom{i+1-m}{g}B_g\\
=&\frac{(-1)^{i+1}}{(i+1)2^{i+1}}\sum_{g=0}^i 2^g\binom{i+1}{g}B_g\\
+&\frac{(-1)^{i+1}}{(i+1)2^{i+1}} \sum_{1\le m}(-1)^m p^{m}\binom{i+1}{m}\sum_{g=0}^i 2^g\binom{i+1-m}{g}B_g\\
=&\frac{(-1)^{i+1}}{(i+1)2^{i+1}}\sum_{g=0}^{i+1} 2^g\binom{i+1}{g}B_g -\frac{(-1)^{i+1}}{(i+1)} B_{i+1}  \\
-&\frac{(-1)^{i+1}}{(i+1)2^{i+1}} \sum_{0\le m}(-1)^m p^{m+1}\binom{i+1}{m+1}\sum_{g=0}^{i-m} 2^g\binom{i-m}{g}B_g.
   \end{align*}
   
\noindent Then, accounting for \eqref{aze1bis} from Lemma \ref{ww} and rearranging, we have
\begin{equation}\label{4.2}  \sum_{j=1 }^\frac{p-1}{2}j^i=\frac{(-1)^{i+1}}{(i+1)2^{i}}B_{i+1}(1-2^{i+1}) +\sum_{m\ge0}\frac{(-1)^{m+i}p^{m+1}}{(m+1)!}m!\binom{i}{m}B_{i-m}\left(\frac{1}{2^i}-\frac{1}{2^{m+1}} \right).
   \end{equation}

\noindent Let $i=p^{n-1}(p-1)-1$.\\By Euler theorem, $2^i \equiv \frac{1}{2} \bmod p^n$ and $ H_{\frac{p-1}{2}}\equiv \sum_{j=1} ^\frac{p-1}{2}j^{p^{n-1}(p-1)-1} \bmod p^n$, so that
\begin{align*}H_{\frac{p-1}{2}}&\equiv \frac{2pB_{p^{n-1}(p-1)}}{p-1}\frac{1-2^{p^{n-1}(p-1)}}{p^n} \\
&-\sum_{m\ge1}\frac{(-1)^mp^{m}}{(m+1)!}m!\binom{p^{n-1}(p-1)-1}{m}pB_{p^{n-1}(p-1)-1-m}\left(2-\frac{1}{2^{m+1}} \right)\\
&\pmod {p^n}.
\end {align*}

\noindent Note that the sum now starts from $m=1$, because $B_{p^{n-1}(p-1)-1}=0$. Now, we simplify the sum by accounting for Lemma \ref{lemjohn} and Lemma \ref{binom} and we also make use of Lemma \ref{2.5} and  Lemma \ref{2.6}, so that

\begin{align*}H_{\frac{p-1}{2}}&\equiv -2\sum_{j=0}^{n-1}(-1)^j\frac{q_p^{j+1}}{j+1}p^j -2\delta^{n+1}_p q_pp^{n-1}\\
&-\sum_{m\ge1}\frac{p^{m+1}}{m+1}B_{p^{n-1}(p-1)-1-m}\left(2-\frac{1}{2^{m+1}} \right) \pmod {p^n}.
\end {align*}

\noindent Accounting for the vanishing Bernoulli numbers, when $m$ is even in the above congruence (nota: when $m=p^{n-1}(p-1)-2$, the Bernoulli number does not vanish: it is $-\frac{1}{2}$ which is $p$-integral, but since $p^{n-1}(p-1)-1 \ge n$ for $p\ge3$,  the corresponding term in the sum is clearly zero $\bmod  \  p^n$), we can then write

\begin{align*}H_{\frac{p-1}{2}}&\equiv -2\sum_{j=0}^{n-1}\frac{(-p)^jq_p^{j+1}}{j+1} -2\delta^{n+1}_p q_pp^{n-1}-\sum_{m\ge1}\frac{p^{2m}}{2m}B_{p^{n-1}(p-1)-2m}\left(2-\frac{1}{2^{2m}} \right)\\
&\pmod {p^n}.
\end {align*}

\noindent To complete the proof of \eqref{4.1}, we argue that the second sum on the right hand side of the above congruence may be limited to $m<\frac{n+1}{2}$, because
when $m \ge\frac{n+1}{2}$, $\frac{p^{2m}}{2m}$ is divisible by $p^{n+1}$: this is obviously the case when $p$ does not divides $m$. When $p$ divides $m$, let $p^k$ be the largest power of $p$ which divides $m$. Suppose $2m-k <n+1$ then $2p^k-k <n+1$ but by hypothesis $n+1 < 2p $, then $2p-2p^k+k \ge 2 $ and this is not possible when $k>0$.
\end {proof}

\ 

\begin{prop}  Let $p$ be an odd prime. Let $n,h$ integers and $p>\frac{n+1}{2}> h\ge 1$. The following congruence holds:
\begin{equation}\label{4.3} \begin{split}
H^{2h+1}_{\frac{p-1}{2}}&+ (2^{2h+1}-2)\frac{B_{p^{n-1}(p-1)-2h}}{2h} \\
&+\sum_{i=1}^\frac{n-1}{2}p^{2i}\binom{2i+2h}{2i}\frac{B_{p^{n-1}(p-1)-2(h+i)}}{2(h+i)}\left(2^{2h+1}-\frac{1}{2^{2i}} \right) \equiv0 \pmod {p^{n-1}}.
\end{split}
\end{equation}

\begin {proof}We start from \eqref{4.2}, with $i=p^{n-1}(p-1)-2h-1$. Again by Euler theorem, we have  $2^i \equiv \frac{1}{2^{2h+1}} \bmod p^n$, $ H^{2h+1}_{\frac{p-1}{2}}\equiv \sum_{j=1}^\frac{p-1}{2}j^{p^{n-1}(p-1)-2h-1} \bmod p^n$  and $2^{p^{n-1}(p-1)-2h} \equiv 2^{-2h}\bmod p^n $, so that, by the same argumentation as in the proof of the previous proposition, we have\\  
\begin{align*}H^{2h+1}_{\frac{p-1}{2}}&\equiv \frac{2^{2h+1}(1-2^{-2h})}{p^{n-1}(p-1)-2h}B_{p^{n-1}(p-1)-2h} \\
&-\sum_{m\ge1}\frac{p^{m+1}}{m+1}\binom{m+2h}{m}B_{p^{n-1}(p-1)-2h-1-m}\left(2^{2h+1}-\frac{1}{2^{m+1}} \right) \pmod {p^n}.
\end {align*}
We now account for the vanishing Bernoulli numbers, when $m$ is even in the above congruence. Nota: when $m=p^{n-1}(p-1)-2h-2$, the Bernoulli number is $-\frac{1}{2}$ (which is $p$-integral), but in this case, since $n\ge 2h$ by hypothesis, we have  $p^{n-1}(p-1)-2h-1 \ge p^{n-1}(p-1)-n-1 \ge n+1$ for $p\ge3$ and $n \ge 2$. Moreover $m+1= p^{n-1}(p-1)-2h-1 \equiv -(2h+1) \bmod p$, therefore, since $h<p$ by hypothesis, $p$ may divide $m+1$ only when $h= \frac{p-1}{2}$, and $p$ is the highest power of $p$ which may divide $m+1$. Then the term of index $m=p^{n-1}(p-1)-2h-2$ in the sum is zero $\bmod  \  p^n$. We can then write
\begin{align*}H^{2h+1}_{\frac{p-1}{2}}&\equiv \frac{2^{2h+1}(1-2^{-2h})}{p^{n-1}(p-1)-2h}B_{p^{n-1}(p-1)-2h} \\
&-\sum_{m\ge1}\frac{p^{2m}}{2m}\binom{2m+2h-1}{2m-1}B_{p^{n-1}(p-1)-2(h+m)}\left(2^{2h+1}-\frac{1}{2^{2m}} \right) \pmod {p^n}\\
&\equiv \frac{2^{2h+1}(1-2^{-2h})}{p^{n-1}(p-1)-2h}B_{p^{n-1}(p-1)-2h} \\
&-\sum_{m\ge1}p^{2m}\binom{2m+2h}{2m}\frac{B_{p^{n-1}(p-1)-2(h+m)}}{2(h+m)}\left(2^{2h+1}-\frac{1}{2^{2m}} \right) \pmod {p^n}.
\end {align*}
Reducing $\bmod \  p^{n-1}$, we obtain
\begin{align*}H^{2h+1}_{\frac{p-1}{2}}&\equiv \frac{2-2^{2h+1}}{2h}B_{p^{n-1}(p-1)-2h} \\
&-\sum_{m\ge1}p^{2m}\binom{2m+2h}{2m}\frac{B_{p^{n-1}(p-1)-2(h+m)}}{2(h+m)}\left(2^{2h+1}-\frac{1}{2^{2m}} \right) \pmod {p^{n-1}}.
\end {align*}
To complete the proof of \eqref{4.3}, we argue that the second sum on the right hand side of the above congruence may be limited to $m\le\frac{n-1}{2}$, because the summands may also be written like 
$$ \frac{p^{2m}}{2m}\binom{2m+2h-1}{2m-1}B_{p^{n-1}(p-1)-2(h+m)}\left(2^{2h+1}-\frac{1}{2^{2m}} \right) $$
and when $m>\frac{n-1}{2}$, $\frac{p^{2m}}{2m}$ is divisible by $p^{n}$: this is obviously the case when $p$ does not divides $m$. When $p$ divides $m$, let $p^k$ be the largest power of $p$ which divides $m$. Suppose $2m-k <n$ then $2p^k-k <n$ but by hypothesis $n+1 < 2p $, so $2p-2p^k+k \ge 3$. This is impossible when $k>0$.
\end {proof}
\end{prop}

\ \\


\noindent We are now ready for our final result:
\begin{thm}\label{last}  Let $n \ge 1$ be an integer, $p$ an odd prime, $p  > \frac{ n+1}{2}$. We have
\begin{equation}\label{ee20}\sum_{j=0}^{n-1}\frac{(2^{j+1}-1)}{(j+1)}\frac{(2^{j+2}-1)}{(j+2)}\frac{B_{j+2}}{2^{j-1}}H^{(j+1)}_{\frac{p-1}{2}}p^j+\sum_{j=0}^{n-1}(-1)^j\frac{q_p^{j+1}}{j+1}p^j\equiv 0 \pmod {p^n}. \end{equation}
\end{thm}
\noindent \text{\bf{Remark.}}
	Eisenstein and E. Lehmer congruences \eqref{1.3} and \eqref{1.4} are particular cases of \eqref{ee20}  for $n=1$ and $n=2$, respectively. Next, for $n=3$, we obtain
	\begin{align*}  H_\frac{p-1}{2} -\frac{7}{24}H^{(3)}_\frac{p-1}{2}p^2+2\left(q_p-\frac{q^2_p}{2}p+\frac{q^3_p}{3}p^2\right)& \equiv 0 \pmod{p^3}. \end{align*} By taking \eqref{e9bbs} into account, we recover a congruence originally given by Z.-H. Sun (\cite{Sun00}, Theorem 5.2(c)): 
\begin{align*}  H_\frac{p-1}{2} +\frac{7}{12}B_{p-3}\  p^2+2\left(q_p-\frac{q^2_p}{2}p+\frac{q^3_p}{3}p^2 \right)& \equiv 0 \pmod{p^3}. \end{align*}
Next
\begin{align*}  H_\frac{p-1}{2} -\frac{7}{24}H^{(3)}_\frac{p-1}{2}p^2+2\left(q_p-\frac{q^2_p}{2}p+\frac{q^3_p}{3}p^2 -\frac{q^4_p}{4}p^3\right)& \equiv 0  \pmod{p^4}. \end{align*}
Next, under the condition $p\ge5$, we have
\begin{align*}
H_\frac{p-1}{2}-\frac{7}{24}H^{(3)}_\frac{p-1}{2}p^2+\frac{31}{80}H^{(5)}_\frac{p-1}{2}p^4+2\left(q_p-\frac{q^2_p}{2}p+\frac{q^3_p}{3}p^2 -\frac{q^4_p}{4}p^3+\frac{q^5_p}{5}p^4\right)&\equiv0\\
H_\frac{p-1}{2}-\frac{7}{24}H^{(3)}_\frac{p-1}{2}p^2-\frac{93}{40}B_{p-5}p^4+2\left(q_p - \frac{q^2_p}{2}p+\frac{q^3_p}{3}p^2 -\frac{q^4_p}{4}p^3+\frac{q^5_p}{5}p^4\right)&\equiv0\\
&\pmod{p^5}\\
&\text{etc.}
\end{align*}
 
\noindent \text{\bf{Remark.}} The condition $p > \frac{n+1}{2}$ is sharp. For instance, for $n=5$ and $p=3$, we have $q_p=1$, $H^{(j)}_\frac{p-1}{2}=1$ and the left hand side of \eqref{ee20} is not divisible by $3^5$. Indeed, we check that $1 -\frac{7\cdot 3^2}{24}+\frac{31\cdot 3^4}{80}+2\left(1-\frac{3}{2}+\frac{3^2}{3} -\frac{3^3}{4}+\frac{3^4}{5}\right) =\frac{4293}{80}=\frac{3^4 \cdot 53}{5\cdot 2^4}$.\\

\begin {proof}[Proof of Theorem \ref{last}] It suffices to prove \eqref{ee20} when $n$ is even, because then it is true for $n-1$ since the first sum in \eqref{ee20} is actually limited to $n-2$ when $n$ is even ($B_{n+1}=0$) and $\frac{q_p^{n}}{n}p^{n-1}\equiv 0 \bmod p^{n-1}$ because $p$ does not divide $n$ otherwise $p \le\frac{n}{2}$ which is excluded by hypothesis. So from now on, it is supposed that $n$ is even. For  $h\ge 1$ an integer we define $Z_{n,h}:= \frac{B_{p^{n-1}(p-1)-2h}}{2h}$ and $A_h:=4 \frac{(2^{2h+2}-1)B_{2h+2}}{(2h+1)(2h+2)}$. Congruences \eqref{4.1} and \eqref{4.3} now read, respectively,
\begin{equation}\label{4.5}H_{\frac{p-1}{2}}+2\sum_{j=0}^{n-1}(-1)^j\frac{q_p^{j+1}}{j+1}p^j+2\delta^{n+1}_p q_pp^{n-1}
+\sum_{i=1}^\frac{n}{2}Z_{n,i}(2^{2i+1}-1)\left(\frac{p}{2}\right)^{2i}\equiv 0 \bmod p^n,\end{equation}
\begin{equation}\label{4.6}H^{(2h+1)}_{\frac{p-1}{2}}+Z_{n,h}(2^{2h+1}-2)+\sum_{i=1}^\frac{n-2}{2}\binom{2h+2i}{2h}Z_{n,h+i}(2^{2(h+i)+1}-1)\left(\frac{p}{2}\right)^{2i} \equiv 0 \bmod p^{n-1}.\end{equation}

\noindent Comparing \eqref{4.5} to \eqref{ee20}, eliminating the terms with powers of the Fermat quotient, rearranging and using $h$ as summation index, we see that we need to prove that
\begin{equation}\label{4.7}\begin{split}
\left(2q_p\delta^{n+1}_p +p Z_{n,\frac{n}{2}}\frac{2^{n+1}-1}{2^n}\right)p^{n-1}\equiv  \sum_{h=1}^\frac{n-2}{2}(2^{2h+1}-1)\left(\frac{p}{2}\right)^{2h}\left(A_hH^{(2h+1)}_{\frac{p-1}{2}}-Z_{n,h}\right)&\\
\bmod { \ p^{n}}.\end{split}\end{equation}

\noindent Now \eqref{4.6} may be rearranged, so that
\begin{align*}H^{(2h+1)}_{\frac{p-1}{2}} \equiv Z_{n,h}
-\sum_{i=0 }^\frac{n-2}{2}\binom{2h+2i}{2h}Z_{n,h+i}(2^{2(h+i)+1}-1)\left(\frac{p}{2}\right)^{2i}  \pmod {p^{n-1}}.
\end{align*}
In Congruence \eqref{4.7}, $H^{(2h+1)}_{\frac{p-1}{2}}$ may be replaced by the right hand side of the latter congruence because (i) $(2^{2h+2}-1)B_{2h+2}$ is $p$-integral, by the Von Staudt-Clausen theorem, and (ii)
$\frac{p^{2h}}{(2h+1)(2h+2)} \equiv 0 \bmod p  $ when $h\ge 1$,  since $\frac{p^{2h}}{(2h+1)(2h+2)}=\frac{p^{2h}}{(2h+1)!}\frac{(2h)!}{2(h+1)}$, where the second factor is clearly an integer while the first one is $0  \bmod p  $ by Lemma \ref{lemjohn}. \\
\ \\
Let $L$ (resp. $R$) be the left (resp. right) hand side of \eqref{4.7}, then we have
\begin{align*}
R &\equiv   \sum_{h=1}^\frac{n-2}{2}(2^{2h+1}-1)\left(\frac{p}{2}\right)^{2h}(A_h -1)Z_{n,h}  \\
& -\sum_{h=1}^\frac{n-2}{2}\sum_{i=0}^\frac{n-2}{2}A_h(2^{2h+1}-1)\binom{2h+2i}{2i}Z_{n,h+i}(2^{2(h+i)+1}-1)\left(\frac{p}{2}\right)^{2i+2h} \pmod {p^n}\\
&\equiv  \sum_{h=1}^\frac{n-2}{2}(2^{2h+1}-1)\left(\frac{p}{2}\right)^{2h}(A_h -1)Z_{n,h}  \\
& -\sum_{h=1}^\frac{n-2}{2}\sum_{i=1}^\frac{n-2}{2}A_h(2^{2h+1}-1)\binom{2h+2i}{2i}Z_{n,h+i}(2^{2(h+i)+1}-1)\left(\frac{p}{2}\right)^{2i+2h} \\
& -\sum_{h=1}^\frac{n-2}{2}A_h(2^{2h+1}-1)Z_{n,h}(2^{2h+1}-1)\left(\frac{p}{2}\right)^{2h} \pmod {p^n}.
\end{align*}

\noindent Let $k=i+h$, and we re-index the sums in the second line with $h$ and $k$. $k$ runs from $2$ to $n-2$ and since $1\le i,h \le \frac{n-2}{2}$, $h$ must also satisfy $k-\frac{n-2}{2} \le h \le k-1$ so that actually $h$ runs from $\max(1,k-\frac{n-2}{2})$ to $\min(\frac{n-2}{2},k-1)$, and then
\begin{align*}
R &\equiv   \sum_{h=1}^\frac{n-2}{2}(2^{2h+1}-1)\left(\frac{p}{2}\right)^{2h}(A_h -1)Z_{n,h}  \\
& -\sum_{k=2}^{n-2}Z_{n,k}(2^{2k+1}-1)\left(\frac{p}{2}\right)^{2k}\sum_{h=\max(1,k-\frac{n-2}{2})}^{\min(\frac{n-2}{2},k-1)}\binom{2k}{2h}A_h(2^{2h+1}-1)\\
& -\sum_{h=1}^\frac{n-2}{2}A_h(2^{2h+1}-1)Z_{n,h}(2^{2h+1}-1)\left(\frac{p}{2}\right)^{2h} \pmod {p^n}.
\end{align*}
\noindent We now argue that, in the above congruence, $k$ may be limited to $k\le\frac{n}{2}$. Indeed, $p$ may divide the denominator of $\binom{2k}{2h}Z_{n,k}$ once at most, because $$\binom{2k}{2h}Z_{n,k}=\frac{1}{2h}\binom{2k-1}{2h-1}B_{p^{n-1}(p-1)-2k}$$ and $p$ does not divide $h$ because $p >\frac{n+1}{2} \ge h+\frac{3}{2}$. On the other hand, $p$ may also divide the denominator of $4 \frac{(2^{2h+2}-1)B_{2h+2}}{(2h+1)(2h+2)}$ once at most, since $(2^{2h+2}-1)B_{2h+2}$ is $p$-integral, $p$ does not divide $h+1$, since $p>h+1$ and if we suppose $p^g$ is the largest power of $p$ that divides $2h+1$, then $p^g \le 2h+1 \le n-2+1=n-1<2p-1-1$ that is $p^g <2p-2$, which is possible only if $g\le1$. Then
\begin{align*}
R &\equiv   \sum_{k=1}^\frac{n-2}{2}(2^{2k+1}-1)\left(\frac{p}{2}\right)^{2k}(A_k -1)Z_{n,k}\\
&-\sum_{k=2}^{\frac{n}{2}}Z_{n,k}(2^{2k+1}-1)\left(\frac{p}{2}\right)^{2k}\sum_{h=1}^{k-1}\binom{2k}{2h}A_h(2^{2h+1}-1)\\
& -\sum_{k=1}^\frac{n-2}{2}A_k(2^{2k+1}-1)Z_{n,k}(2^{2k+1}-1)\left(\frac{p}{2}\right)^{2k} \pmod {p^n}.
\end{align*}
\noindent We now evaluate the inner sum in the second line of the above congruence:
\begin{align*} \sum_{h=1}^kA_h(2^{2h+1}-1)\binom{2k}{2h}&=\sum_{h=1}^k4 \frac{(2^{2h+2}-1)B_{2h+2}}{(2h+1)(2h+2)}(2^{2h+1}-1)\binom{2k}{2h}\\
&=\sum_{h=0}^k4 \frac{(2^{2h+2}-1)B_{2h+2}}{(2h+1)(2h+2)}(2^{2h+1}-1)\binom{2k}{2h}-4 \frac{(2^{2}-1)B_{2}}{2}.\end{align*}

\noindent That is
\begin{align*}
\sum_{h=1}^kA_h(2^{2h+1}-1)\binom{2k}{2h}&=\sum_{h=0}^k4 \frac{(2^{2h+2}-1)B_{2h+2}}{(2h+1)(2h+2)}(2^{2h+1}-1)\binom{2k}{2h}-1\\
&=\frac{1}{2k+1}\sum_{h=0}^k4 \frac{(2^{2h+2}-1)B_{2h+2}}{(2h+2)}(2^{2h+1}-1)\binom{2k+1}{2h+1}-1.
\end{align*}

\noindent On the right hand side of the last line, let $j= 2h+1$. $j$ takes all odd values from $1$ to $2k+1$, but we may let $j$ take all values from $0$ to $2k+1$, because the summands corresponding to even $j$ are zero, either because $2^0 -1 =0$ or because of vanishing Bernoulli numbers.\\
\ \\
\noindent Then

\begin{align*}\sum_{h=1}^kA_h(2^{2h+1}-1)\binom{2k}{2h}&=\frac{4}{2k+1}\sum_{j=0}^{2k+1} \frac{B_{j+1}}{(j+1)}(2^{j}-1)(2^{j+1}-1)\binom{2k+1}{j}-1 \\
&=\frac{4}{2k+1}\frac{B_{2k+2}}{2k+2}(2^{2k+2}-1)-1    \text{ \ \ \ by Lemma \ref{2.3}}\\
&=A_k-1.\end{align*}
Then \begin{align*}
\sum_{h=1}^{k-1}A_h(2^{2h+1}-1)\binom{2k}{2h}&=A_k-1-A_k(2^{2k+1}-1).
\end{align*}
\noindent And the latter may now be used into $R$, so that
\begin{align*}  
R&\equiv   \sum_{k=1}^\frac{n-2}{2}(2^{2k+1}-1)\left(\frac{p}{2}\right)^{2k}(A_k -1)Z_{n,k}  \\
& -\sum_{k=2}^\frac{n}{2}Z_{n,k}(2^{2k+1}-1)\left(\frac{p}{2}\right)^{2k}\left(A_k-1-A_k(2^{2k+1}-1)\right)\\
& - \sum_{k=1}^\frac{n-2}{2}A_k(2^{2k+1}-1)Z_{n,k}(2^{2k+1}-1)\left(\frac{p}{2}\right)^{2k} \pmod {p^n}. \end{align*}
That is 
\begin{align*}
R&\equiv  \sum_{k=1}^\frac{n-2}{2}(2^{2k+1}-1)\left(\frac{p}{2}\right)^{2k}(A_k -1)Z_{n,k}  \\
& -\sum_{k=1}^\frac{n-2}{2}Z_{n,k}(2^{2k+1}-1)\left(\frac{p}{2}\right)^{2k}\left(A_k-1-A_k(2^{2k+1}-1)\right)\\
& -\sum_{k=1}^\frac{n-2}{2}A_k(2^{2k+1}-1)Z_{n,k}(2^{2k+1}-1)\left(\frac{p}{2}\right)^{2k} \\
& +7Z_{n,1}\left(\frac{p}{2}\right)^{2}\left(-6A_1-1\right)-Z_{n,\frac{n}{2}}(2^{n+1}-1)\left(\frac{p}{2}\right)^{n}\left(2A_{\frac{n}{2}}(1-2^n)-1\right)\pmod {p^n}.
\end{align*} 

\noindent This can be considerably simplified, since all the sums cancel out and $A_1=-\frac{1}{6}$, so that we are left with
 \begin{align*}  
R&\equiv pZ_{n,\frac{n}{2}}\frac{2^{n+1}-1}{2^n}\left(2A_\frac{n}{2}(2^{n}-1)+1\right) p^{n-1}  \pmod {p^n}.  \end{align*} 
Recall that in order to prove our theorem we just need to show that $L \equiv R \bmod p^n$. So there merely remains to prove that
\begin{align} \label{pp} q_p\delta^{n+1}_p \equiv pZ_{n,\frac{n}{2}}\frac{(2^{n+2}-1)(2^{n+1}-1)(2^{n}-1)B_{n+2}}{2^{n-2}(n+1)(n+2)}\pmod p.  \end{align}
This is obtained as follows. Recall that $ (2^{n+2}-1)B_{n+2}$ is $p$-integral, and first consider the case where $p-1$ does not divide $n$, then $Z_{n,\frac{n}{2}}$ is $p$-integral since $p$ does not divide $n$ otherwise $p \le \frac{n}{2}$ which is excluded by hypothesis. Moreover, $p$ does not divide $n+2$ otherwise $p \le \frac{n}{2}+1$ and $ \frac{n+1}{2} <p$, then $p= \frac{n}{2}+1$, this is impossible since $p-1$ does not divide $n$; and $p$ does not divide $n+1$ either, because $p\le n+1 <2p$ would imply $p=n+1$ which is impossible since $p-1$ does not divide $n$. So we have shown that if $p-1$ does not divide $n$, the right hand side of \eqref{pp} is $0  \bmod p$ and obviously so is the left hand side since in this case $\delta^{n+1}_p =0$. Now, if $p-1$ divides $n$, then since $p>\frac{n+1}{2}$, either $p=n+1$ or $2p=n+2$. If $n=2p-2$, the left hand side of \eqref{pp} is $0$ and by Lemma \ref{2.5}, $pZ_{n,\frac{n}{2}} \equiv  \frac{1}{2} \bmod p$. If $p=3$, $n=4$, the right hand side of \eqref{pp} is $\frac{1}{2}\frac{(2^{6}-1)(2^{5}-1)(2^{4}-1)}{2^{2}(4+1)(4+2)(42)}=\frac{93}{32} \equiv 0 \bmod 3$. If $p\neq 3$, by Kummer congruence \eqref{e2ez},  $\frac{B_{n+2}}{n+2}\equiv \frac{B_2}{2} \bmod p$ then the  the right hand side of \eqref{pp} is $\frac{1}{2}\frac{(2^{2p}-1)(2^{2p-1}-1)(4^{p-1}-1)}{2^{2p-4}(2p-1)12}$  which is $0  \bmod p$, by Fermat little theorem.
If $n=p-1$,  the left hand side of \eqref{pp} is $q_p$ and by Lemma \ref{2.5}, $pZ_{n,\frac{n}{2}} \equiv 1\bmod p$. If $p=3$, $n=2$, $q_3=1$ and the right hand side of \eqref{pp} is $-\frac{(2^{4}-1)(2^{3}-1)(2^{2}-1)}{(2+1)(2+2)(30)}=\frac{-7}{8} \equiv 1= q_3\bmod 3$.  If $p\neq 3$, by Kummer congruence \eqref{e2ez},  $\frac{B_{n+2}}{n+2}\equiv \frac{B_2}{2} \bmod p$ then the right hand side is 
$\frac{(2^{p+1}-1)(2^{p}-1)(2^{p-1}-1)}{2^{p-3}\cdot p\cdot 12} \equiv \frac{2^{p-1}-1}{p} =q_p  \bmod p $.
\end{proof}


\end{document}